\documentclass[a4paper,12pt,reqno]{amsart}
\usepackage[margin=3cm]{geometry}
\usepackage{setspace}
\usepackage{mathtools}
\usepackage{amssymb, amsfonts, mathrsfs}

\usepackage{aliascnt}
\usepackage{booktabs}
\usepackage{ragged2e}
\usepackage{xltabular}
\newcolumntype{C}{>{\tiny}c}
\newcolumntype{R}{>{\tiny}r}
\newcolumntype{B}{>{\tiny\RaggedRight\arraybackslash}X}
\usepackage{listings}
\usepackage{caption}
\newtheorem{theorem}{Theorem}[section]
\newaliascnt{lemma}{theorem}

\aliascntresetthe{lemma}
\newaliascnt{proposition}{theorem}
\newtheorem{proposition}[proposition]{Proposition}
\aliascntresetthe{proposition}
\newaliascnt{corollary}{theorem}
\newtheorem{corollary}[corollary]{Corollary}
\aliascntresetthe{corollary}
\newaliascnt{conjecture}{theorem}

\aliascntresetthe{conjecture}

\theoremstyle{definition}
\newaliascnt{definition}{theorem}
\newtheorem{definition}[definition]{Definition}
\aliascntresetthe{definition}
\newaliascnt{claim}{theorem}

\aliascntresetthe{claim}
\newaliascnt{example}{theorem}
\newtheorem{example}[example]{Example}
\aliascntresetthe{example}
\newaliascnt{notation}{theorem}

\aliascntresetthe{notation}
\newaliascnt{question}{theorem}
\newtheorem{question}[question]{Question}
\aliascntresetthe{question}

\theoremstyle{remark}
\newaliascnt{remark}{theorem}
\newtheorem{remark}[remark]{Remark}
\aliascntresetthe{remark}

\numberwithin{equation}{section}
\numberwithin{table}{section}
\AtBeginDocument{\numberwithin{lstlisting}{section}}

\usepackage[hidelinks]{hyperref} 
\usepackage[capitalize]{cleveref}

\newcommand{\bC}{{\mathbb C}}
\newcommand{\bQ}{{\mathbb Q}}
\newcommand{\bP}{{\mathbb P}}

\newcommand\OO{{\mathcal{O}}}
\newcommand\OX{{\mathcal{O}_X}}

\newcommand\bB{\mathbb{B}}

\newcommand\rank{{\text{\rm rank}}}

\newcommand{\lsgeq}{\succcurlyeq}
\newcommand{\lsleq}{\preccurlyeq}

\newcommand{\lsg}{\succ}

\newcommand{\Mov}{\operatorname{Mov}}

\newcommand{\Supp}{\operatorname{Supp}}
\newcommand{\Exc}{\operatorname{Exc}}

\title[On generic finiteness of threefolds]{On generic finiteness of pluricanonical maps of threefolds of general type}

\author{Tianyue Zhang}
\address{School of Mathematical Sciences, Fudan University, Shanghai 200433, China}
\email{22110180056@m.fudan.edu.cn}


\begin{document}

\begin{abstract}
    We prove that $|6K_X|$ defines a generically finite map for all minimal 3-folds $X$ of general type with $P_2(X)\geq 2$, which is optimal. We also prove that $|nK_X|$ defines a generically finite map for all minimal 3-folds $X$ of general type when $n\geq 38$. The essential technical ingredients of this paper are a new generic finiteness criterion for surfaces and an effective comparison inequality under a special resolution.
\end{abstract}

\maketitle
\tableofcontents

\section{Introduction}

Classifying varieties is one of the most fundamental problems in algebraic geometry. By criteria for embedding a polarized variety into a weighted projective space in \cite{Noether2025}, generic finiteness of pluricanonical maps plays an important role in constructing the moduli spaces of designated classes of 3-folds. Thus the study of generic finiteness of pluricanonical maps is of great value for classifying varieties.

There are many classical works in this aspect. For a minimal surface $S$ of general type, Enrico Bombieri, Yoichi Miyaoka and Fabrizio Catanese proved that $|nK_S|$ defines a generically finite map when $n\geq 3$ (\cite[Main Theorem]{Bombieri}, \cite[Theorem 3]{Miyaoka1976}, \cite{BombieriCatanese1978}, \cite{Catanese1981}); Gang Xiao proved that $|2K_S|$ defines a generically finite map if and only if $S$ is not a $(1,0)$-surface (\cite[Theorem 1]{X2}). For a minimal 3-fold $X$ of general type, Kan Wu proved that $|3K_X|$ defines a generically finite map when $p_g(X)\geq 4$ (\cite[Theorem 1.1]{Wu2019}); Meng Chen, Chen Jiang and Jianshi Yan proved that $|2K_X|$ is not composed with a pencil when $p_g(X)\geq 202$ (\cite[Theorem 1.2]{CJY2026}).

There are many classical works about birationality in parallel. For a minimal surface $S$ of general type, Enrico Bombieri proved that $|nK_S|$ defines a birational map when $n\geq 5$ (\cite[Main Theorem]{Bombieri}). For a minimal 3-fold $X$ of general type, Meng Chen proved that $|nK_X|$ defines a birational map when $P_m(X)\geq 2$ and $n\geq 5m+6$ (\cite[Theorem 0.1]{Chen2004}); Meng Chen and Jungkai A. Chen proved that $|nK_X|$ defines a birational map when $P_2(X)\geq 2$ and $n\geq 11$ (\cite[Theorem 1.8]{EXP3}); Meng Chen and Jungkai A. Chen proved that $|nK_X|$ defines a birational map when $n\geq 57$ (\cite[Theorem 1.1]{Chen2018}, \cite[Theorem 6.2]{EXP3}).

Our main results are the following three theorems.

\begin{theorem}\label{thm6K}
    Let $X$ be a minimal 3-fold of general type with $P_2(X)\geq 2$. Then $|6K_X|$ defines a generically finite map.
\end{theorem}

\begin{theorem}\label{thm3m2}
    Let $X$ be a minimal 3-fold of general type with $P_m(X)\geq 2$. Then $|nK_X|$ defines a generically finite map when $n\geq 3m+2$.
\end{theorem}

\begin{theorem}\label{thm38}
    Let $X$ be a minimal 3-fold of general type. Then $|nK_X|$ defines a generically finite map when $n\geq 38$.
\end{theorem}

The following two examples show that both the integer 6 and the integer 2 are optimal in \cref{thm6K}.

\begin{example}[{\cite[II.5.4]{IF2000}}]
    $X_{6,18}$ in $\bP(2,2,3,3,4,9)$ is a minimal 3-fold of general type such that $P_2=2$ and $|nK|$ does not define a generically finite map for all $n\leq 5$.
\end{example}

\begin{example}[{\cite[II.5.1]{IF2000}}]
    $X_{30}$ in $\bP(2,3,4,5,15)$ is a minimal 3-fold of general type such that $P_2=1$ and $|nK|$ does not define a generically finite map for all $n\leq 7$.
\end{example}

There are many challenges in the study of the $P_2\geq 2$ case since we no longer have $p_g>0$. Moreover, the failure of generic finiteness of $\Phi_{|2K|}$ of $(1,0)$-surfaces causes many difficulties. To characterize generic finiteness precisely, we establish a new generic finiteness criterion for surfaces, and we prove an effective comparison inequality using a special resolution which is a central new technique in \cite{Noether2025}.

When studying generic finiteness of 3-folds with large pluricanonical section index, we establish a better non-vanishing theorem than \cite[Theorem 1.1 (1)]{EXP2}, and we prove several criteria to trace the relationship between $K^3$ and $P_n$ carefully. Moreover, we introduce a modification of baskets that helps us focus on a finite number of possibilities.

The structure of this paper is as follows. In \cref{secP}, we gather some basic definitions and theorems that will be used. \cref{secSurf}, \cref{secCompare} and \cref{secThm6K} are devoted to proving \cref{thm6K}. \cref{secNonV}, \cref{secCriteria3}, \cref{secDelta13} and \cref{secThm38} are devoted to proving \cref{thm3m2} and \cref{thm38}.

\section{Preliminaries}\label{secP}

Throughout this paper, we work over the complex number field $\bC$, and we will freely use standard definitions and theorems in \cite{KMM, KM1998}.

A \textit{variety} means an integral separated scheme of finite type over $\bC$. A \textit{minimal} variety means a $\bQ$-factorial terminal projective variety with nef canonical divisor.

A \textit{divisor} means a Weil divisor. A $\bQ$-divisor means a $\bQ$-linear combination of divisors. For $\bQ$-divisors $D_1$ and $D_2$, we write $D_1\geq D_2$ if $D_1-D_2$ is an effective $\bQ$-divisor; we write $D_1\geq_\bQ D_2$ if $D_1-D_2$ is $\bQ$-linearly equivalent to an effective $\bQ$-divisor.

For linear systems $\Lambda_1$ and $\Lambda_2$, we write $\Lambda_1\lsgeq \Lambda_2$ if $\Lambda_1\supseteq \Lambda_2+Z$ where $Z$ is an effective divisor.

An $(a,b)$-surface means a smooth projective surface $S$ of general type with $K_{S_0}^2=a$ and $p_g(S)=b$ where $S_0$ is the minimal model of $S$.

For a linear system $\Lambda$ on a normal projective variety, $\Phi_\Lambda$ means the rational map defined by $\Lambda$.

The following two theorems are classical results by Gang Xiao.

\begin{theorem}[{\cite[Theorem 1]{X2}, \cite[Theorem 0.1]{CW}}]\label{thmXG1}
    Let $S$ be a minimal surface of general type. Then $|2K_S|$ defines a generically finite map if and only if $S$ is not a $(1,0)$-surface.
\end{theorem}

\begin{theorem}[{\cite[Lemma 8]{X2}, \cite[Lemma 1.3]{CW}}]\label{thmXG2}
    Let $S$ be a minimal surface of general type with $q(S)=0$ and $K_{S}^2\leq 2$. Let $\theta$ be a torsion line bundle on $S$. Then $h^1(S,\theta)=0$.
\end{theorem}

The following theorem is a classical result by Enrico Bombieri, Yoichi Miyaoka and Fabrizio Catanese.

\begin{theorem}[{\cite[Main Theorem]{Bombieri}, \cite[Theorem 3]{Miyaoka1976}, \cite{BombieriCatanese1978}, \cite{Catanese1981}}]\label{thmSurf3K}
    Let $S$ be a minimal surface of general type. Then $|nK_S|$ defines a generically finite map when $n\geq 3$.
\end{theorem}

The following theorem is a classical result by Meng Chen, Jheng-Jie Chen, Jungkai A. Chen, Zhi Jiang and Christopher D. Hacon.

\begin{theorem}[{\cite[Theorem 1.3]{CCCJ}, \cite[Theorem 1.1]{CCJ}, \cite[Proposition 2.9]{CH}}]\label{thmIrr5}
    Let $X$ be an irregular minimal 3-fold of general type. Then $|nK_X|$ defines a birational map when $n\geq 5$.
\end{theorem}

\begin{definition}[{\cite[Remark 3.2]{Chen2001}}]
    Let $X$ be a normal projective variety on which a linear system $\Lambda$ is of dimension $p\geq 1$. Let $f:W\to X$ be a resolution of singularities such that $\Mov f^*\Lambda$ is base point free. Let $W\to B\to \bP^p$ be the Stein factorization of $\Phi_{\Mov f^*\Lambda}$.

    \begin{itemize}
        \item If $\Lambda$ is not composed with a pencil (namely, $\dim\overline{\Phi_\Lambda(X)} \geq 2$), then a \textit{generic irreducible element} of $f^*\Lambda$ means a general element in $\Mov f^*\Lambda$ which is smooth and irreducible by Bertini's theorem.

        \item If $\Lambda$ is composed with a pencil, then a \textit{generic irreducible element} of $f^*\Lambda$ means a general fiber $F$ of $W\to B$ which is smooth and irreducible. By definition of $\Phi_{\Mov f^*\Lambda}$, $F$ is an irreducible component of a general element $M\in \Mov f^*\Lambda$ and $M\equiv aF$ where $a\geq p$. If the pencil is a rational pencil (namely, $B\cong \bP^1$), then $M\sim aF$.
    \end{itemize}
\end{definition}

The following theorem is Tankeev's principle (\cite[Lemma 2]{Tankeev1971}) for generic finiteness.

\begin{theorem}\label{thmTankeev}
    Let $X$ be a smooth projective variety on which $Y$ is a prime divisor. Let $V$ be a linear system on $X$. If $V|_Y$ defines a generically finite map and $Y\leq D$ for some $D\in V$, then $V$ defines a generically finite map.
\end{theorem}

\begin{proof}
    Since $V|_Y$ defines a generically finite map, we have
    \begin{align*}
        \dim \overline{\Phi_V(Y)}=\dim Y=\dim X-1.
    \end{align*}

    Since $Y\leq D$, we have
    \[
        \overline{\Phi_V(Y)}\subseteq \overline{\Phi_V(D)}.
    \]

    Since $D\in V$, we have
    \[
        \overline{\Phi_V(D)}\subsetneq \overline{\Phi_V(X)}.
    \]

    Thus
    \[
        \dim \overline{\Phi_V(X)}\geq \dim \overline{\Phi_V(D)}+1\geq \dim X.
    \]
\end{proof}

\begin{corollary}\label{coroTankeev1}
    Let $X$ be a smooth projective variety on which $Y$ is a prime divisor. Let $B$ be a divisor on $X$. If $|B||_Y$ defines a generically finite map and $h^0(B-Y)>0$, then $|B|$ defines a generically finite map. \qed
\end{corollary}

\begin{corollary}\label{coroTankeev2}
    Let $X$ be a smooth projective variety on which a base point free linear system $\Lambda$ is composed with a pencil. Let $Y$ be a generic irreducible element of $\Lambda$. Let $B$ be a divisor on $X$. If $|B+2Y||_Y$ defines a generically finite map and $h^1(B)=0$, then $|B+2Y|$ defines a generically finite map.
\end{corollary}

\begin{proof}
    Since $h^1(B)=0$, we have
    \[
        h^0(B+Y)\geq h^0(B|_Y).
    \]

    Since $|B+2Y||_Y$ is non-empty, we have $h^0(B|_Y)>0$. Thus $h^0(B+Y)>0$.

    By \cref{coroTankeev1}, $|B+2Y|$ defines a generically finite map.
\end{proof}

The following proposition can be used to prove effective comparison inequalities.

\begin{proposition}[{\cite[Proposition 3.1]{Noether2025}}]\label{propCompareBirG}
    Let $X$ be a $\bQ$-factorial normal projective variety on which a linear system $\Lambda$ is composed with a pencil. Then there exists a $\bQ$-factorial terminal projective variety $X'$ with a birational morphism $g:X'\to X$ having the following properties:
    \begin{enumerate}
        \item $\Mov g^*\Lambda$ is base point free, and $\Phi_{\Mov g^*\Lambda}=\Phi_{\Lambda}\circ g$;
        \item for a general fiber $F$ of $X'\to B$ where $X'\to B\to \bP^{\dim\Lambda}$ is the Stein factorization of $\Phi_{\Mov g^*\Lambda}$, we have that $K_{X'}+F$ is $g$-nef and
        \[
            g^*(K_X+g_*F)-K_{X'}-F
        \]
        is an effective $g$-exceptional $\bQ$-divisor.
    \end{enumerate}
\end{proposition}

\begin{definition}
    Let $X$ be a $\bQ$-factorial normal projective variety on which a linear system $\Lambda$ is of positive dimension. Let $f:W\to X$ be a resolution of singularities such that the exceptional set $\Exc(f)$ is a simple normal crossing divisor.

    \begin{itemize}
        \item If $\Lambda$ is not composed with a pencil, we define $f$ to be a \textit{Chen resolution} with respect to $\Lambda$ if $\Mov f^*\Lambda$ is base point free.

        \item If $\Lambda$ is composed with a pencil, keep the notation in \cref{propCompareBirG}, we define $f$ to be a \textit{Chen resolution} with respect to $\Lambda$ if $f$ factors through $g$.
    \end{itemize}

    Notice that in each case, Chen resolution always exists and $\Mov f^*\Lambda$ is base point free.
\end{definition}

The following theorem is a Reider type birationality criterion by Vladimir Ma\c{s}ek and Adrian Langer.

\begin{theorem}[{\cite{Masek1999}, \cite[Theorem 0.1]{Langer2001}, \cite[Theorem 2.3]{CCCJ}}]\label{thmMasek}
    Let $S$ be a smooth projective surface. Let $L$ be a nef $\bQ$-divisor on $S$. If $L^2>8$ and
    \[
        (L\cdot C)> \frac{4}{1+\sqrt{1-\frac{8}{L^2}}}
    \]
    for any curve $C$ on $S$ passing through two very general points of $S$, then $|K_S+\lceil L\rceil|$ defines a birational map.
\end{theorem}

\begin{proposition}[{\cite[Lemma 2.5]{EXP3}}]\label{propGeneralKC2}
    Let $S$ be a smooth projective surface of general type. Let $\sigma$ be the morphism from $S$ to its minimal model $S_0$. Let $C$ be a curve on $S$ passing through a very general point of $S$. If $S$ is not a $(1,2)$-surface, then $(\sigma^*K_{S_0}\cdot C)\geq 2$.
\end{proposition}

\begin{proposition}[{\cite[Lemma 2.1]{EXP3}}]\label{propQCompare}
    Let $X$ be a minimal 3-fold of general type on which a linear system $\Lambda\lsleq |mK_X|$ is composed with a pencil. Let $p=\dim\Lambda$. Let $f:W\to X$ be a Chen resolution with respect to $\Lambda$. Let $S$ be a generic irreducible element of $f^*\Lambda$. Let $\sigma$ be the morphism from $S$ to its minimal model $S_0$. Then
    \[
        f^*(K_X)|_S\geq_\bQ \frac{p}{m+p} \sigma^*K_{S_0}.
    \]
\end{proposition}

\begin{definition}[{\cite{EXP1}}]
    A \textit{formal basket} is a finite collection (permitting multiplicity) of pairs of coprime integers $(b,r)$ satisfying $0<2b\leq r$.

    Let $X$ be a terminal projective 3-fold. By Miles Reid (\cite{YPG}), there is a formal basket
    \begin{align*}
        B_X=\{ n_{b,r}\times (b,r) \mid \gcd(b,r)=1,\ 0<2b\leq r \}
    \end{align*}
    associated to $X$, which comes from locally deforming singularities into cyclic quotient singularities, where a pair $(b,r)$ corresponds to a virtual cyclic quotient singularity of type $\frac{1}{r}(1,-1,b)$. We call $B_X$ the \textit{Reid basket} of $X$.

    By Reid's Riemann-Roch formula (\cite[Corollary 10.3]{YPG}), for any positive integer $m$, we have
    \[
        \chi(\OX(mK_X))=\frac{1}{12}m(m-1)(2m-1)K_X^3-(2m-1)\chi(\OX)+l(m)
    \]
    where
    \[
        l(m)=\sum_{\gcd(b,r)=1,\ 0<2b\leq r} n_{b,r} \sum_{j=1}^{m-1} \frac{\overline{jb}(r-\overline{jb})}{2r}.
    \]
    Here $\overline{jb}$ means the smallest non-negative residue of $jb$ modulo $r$.

    A \textit{weighted basket} is a triple $\bB= (B,\chi,\chi_2)$ where $B$ is a formal basket, $\chi$ is an integer and $\chi_2$ is a non-negative integer. We define $K^3(\bB)$, $\chi_3(\bB)$, $\chi_4(\bB),\dots$ to be numbers that satisfy Reid's Riemann-Roch formula formally.

    For formal baskets
    \begin{gather*}
        B_1=\{(b_1,r_1), (b_2,r_2)\}\cup B,\\
        B_2=\{(b_1+b_2,r_1+r_2)\}\cup B
    \end{gather*}
    where $B$ may be empty, we call $B_2$ a \textit{packing} of $B_1$ and write $B_1\lsg B_2$. If $b_1r_2-b_2r_1=1$, then we call $B_1\lsg B_2$ a \textit{prime packing}. A composite of packings is also called a packing. We write $B\lsgeq B'$ if $B\lsg B'$ or $B=B'$.

    For a weighted basket $(B,\chi,\chi_2)$, according to \cite[(2.3)]{EXP1}, there is a \textit{canonical sequence}
    \[
        B^{(0)}=B^{(1)}=B^{(2)}=B^{(3)}=B^{(4)}\lsgeq B^{(5)}\lsgeq B^{(6)}\lsgeq \cdots \lsgeq B^{(n)}=B
    \]
    such that each step is either an equality or a composite of prime packings, and pairs in $B^{(m)}$ are contained in
    \begin{align}\label{eqBasketCS}
        \{(b,r)\mid \gcd(b,r)=1,\ 0<2b\leq r\leq m\}\cup\{(1,r)\mid r\geq m+1\}.
    \end{align}

    By \cite[Lemma 3.6]{EXP1},
    \begin{align}\label{eqBasketK3}
        K^3(B^{(m)},\chi,\chi_2)\geq K^3(B,\chi,\chi_2).
    \end{align}

    By \cite[Lemma 2.16]{EXP1}, for all $2\leq l\leq m+1$,
    \begin{align}\label{eqBasketChi}
        \chi_l(B^{(m)},\chi,\chi_2)=\chi_l(B,\chi,\chi_2).
    \end{align}
\end{definition}

The following three propositions give lower bounds for $K^3$ and can be used to exclude some non-geometric baskets (namely, formal baskets that cannot be associated to some terminal projective 3-fold).

\begin{proposition}[{\cite[Corollary 3.1]{EXP3}}]\label{propK3mn2}
    Let $X$ be a minimal 3-fold of general type with $P_m(X)=2$ and $P_n(X)\geq 2$. If a general fiber of the induced fibration from $\Phi_{|mK_X|}$ is not a $(1,2)$-surface, then
    \[
        K_X^3\geq\min\left\{\frac{(P_n(X)-1)^3}{n(n+P_n(X)-1)^2}, \frac{2}{mn(m+1)}\right\}.
    \]
\end{proposition}

\begin{proposition}[{\cite[Proposition 4.2]{EXP3}}]\label{propK3P22}
    Let $X$ be a minimal 3-fold of general type with $P_2(X)\geq 2$. Then $K_X^3\geq \frac{1}{14}$.
\end{proposition}

\begin{proposition}[{\cite[(3.4), (3.6), (3.7), (3.8)]{EXP3}}]\label{propK3Pencil}
    Let $X$ be a minimal 3-fold of general type on which a linear system $\Lambda\lsleq |mK_X|$ is composed with a pencil. Let $p=\dim\Lambda$. If a general fiber of the induced fibration from $\Phi_{\Lambda}$ is not a $(1,0)$-surface, then
    \[
        K_X^3\geq \frac{4p^3}{m(m+p)(3m+4p)}.
    \]
\end{proposition}

\begin{definition}[{\cite{EXP3}}]
    Let $X$ be a minimal 3-fold of general type. We define the \textit{pluricanonical section index}
    \[
        \delta(X)\coloneq \min\{m \mid P_m(X)\geq 2\}.
    \]
\end{definition}

The following theorem classifies minimal 3-folds of general type with $\delta\geq 13$.

\begin{theorem}[{\cite[Theorem 5.1]{EXP3}, \cite[Corollary 2.10]{ChenJiang2022}}]\label{thmDelta}
    Let $X$ be a minimal 3-fold of general type with $\delta(X)\geq 13$. Then $2\leq \chi(\OX)\leq 3$, $q(X)=p_g(X)=P_2(X)=0$, $\delta(X)\leq 18$, $\delta(X)\neq 16,17$, and the Reid basket of $X$ is one of the baskets listed in \cref{tableDelta13}.
\end{theorem}

\begin{remark}
    In the above theorem, the bound $\chi(\OX)\leq 3$ is obtained from \cite[Corollary 2.10]{ChenJiang2022}, and excludes 15 types of weighted baskets with $\chi\geq 4$ in \cite[Table F-2]{EXP3}.

    Moreover, there are five other types in \cite[Table F-1, Table F-2]{EXP3} that can be excluded.

    The type 5b weighted basket
    \[
        \bB_{5b}=(\{7\times(1,2),(4,9),3\times(2,5),(5,13),4\times(1,3),(4,15)\},3,0)
    \]
    can be excluded since $\chi_8(\bB_{5b})=1$ and $\chi_{16}(\bB_{5b})=0$.

    The weighted basket of type 2, type 1, type 4 or type 5 can be excluded by Miyaoka-Reid inequality (\cite[(2.7)]{EXP2}) which is derived from \cite[Corollary 6.7]{Miyaoka1987} and \cite[Corollary 10.3]{YPG}.
\end{remark}

\begin{corollary}[{\cite[Corollary 5.3]{EXP3}}]\label{coroDelta}
    Let $X$ be a minimal 3-fold of general type with $\delta(X)\geq 13$. Then $P_n(X)>0$ in each of the following two cases:
    \begin{enumerate}
        \item $\delta(X)=13$ and $n\geq 10$;
        \item $n\geq 20$.
    \end{enumerate}
\end{corollary}

\begin{proposition}[{\cite[Lemma 3.2]{EXP2}}]\label{propPChi1}
    Let $X$ be a minimal 3-fold of general type with $\chi(\OX)=1$. Then, for any $m\geq 2$, we have
    \[
        P_{m+2}(X)\geq P_m(X)+P_2(X).
    \]
\end{proposition}

\begin{theorem}[{\cite[Corollary 3.13, Theorem 3.14]{EXP2}}]\label{thmPChi1}
    Let $X$ be a minimal 3-fold of general type with $\chi(\OX)=1$. Then $P_{10}(X)\geq 2$ and $P_n(X)>0$ when $n\geq 7$.
\end{theorem}

\section{A generic finiteness criterion for surfaces}\label{secSurf}

The following generic finiteness criterion for surfaces helps us overcome the difficulty of the failure of generic finiteness of $\Phi_{|2K|}$ of $(1,0)$-surfaces.

\begin{theorem}\label{thmMain1}
    Let $S$ be a smooth projective surface of general type. Let $f$ be the morphism from $S$ to its minimal model $S_0$. Let $B$ be an effective divisor on $S$ such that $f^*K_{S_0}+B$ is effective. Then
    \[
        |K_S+f^*K_{S_0}+B|
    \]
    defines a generically finite map.
\end{theorem}

\begin{proof}
    By \cref{thmXG1}, we only need to consider the case that $S_0$ is a $(1,0)$-surface. Thus $K_{S_0}^2=1$ and $q(S)=p_g(S)=0$.

    Let $D=f^*K_{S_0}+B$. If $B,D$ have a common irreducible component $C$, then we can replace $B,D$ with $B-C,D-C$. Thus we may assume that $B,D$ have no common irreducible components. Then $(B\cdot D)\geq 0$. Thus
    \begin{align}\label{eqD2GeqFD}
        D^2= ((f^*K_{S_0}+B)\cdot D) \geq (f^*K_{S_0}\cdot D).
    \end{align}

    Since $f^*K_{S_0}$ is nef on $S$, we have
    \begin{align}\label{eqFDGeq1}
        (f^*K_{S_0}\cdot D)=(f^*K_{S_0}\cdot(f^*K_{S_0}+B))\geq (f^*K_{S_0})^2=1.
    \end{align}

    For any $f$-exceptional curve $E$, we have
    \[
        (D\cdot E)=((f^*K_{S_0}+B)\cdot E)=(B\cdot E).
    \]
    Since $B,D$ have no common irreducible components, we have that either $(B\cdot E)\geq 0$ or $(D\cdot E)\geq 0$. Thus we always have $(D\cdot E)\geq 0$.

    Write $K_S=f^*K_{S_0}+E_S$. Then
    \begin{align}\label{eqDESGeq0}
        (D\cdot E_S)\geq 0.
    \end{align}

    By Riemann-Roch formula and Serre duality theorem,
    \begin{align*}
        h^0(K_S+D)-h^1(-D)+h^0(-D)&=\frac{1}{2}((K_S+D)\cdot D)+\chi(\OO_S)\\
        &=\frac{1}{2}((K_S+D)\cdot D)+1.
    \end{align*}

    Since $D\geq 0$ and $D\neq 0$, we have $h^0(-D)=0$. Thus
    \begin{align*}
        h^0(K_S+D)\geq \frac{1}{2}((K_S+D)\cdot D)+1.
    \end{align*}

    With \cref{eqDESGeq0},
    \begin{align}\label{eqH0KSDGeq}
        h^0(K_S+D)\geq \frac{1}{2}((f^*K_{S_0}+D)\cdot D)+1
    \end{align}

    With \cref{eqD2GeqFD} and \cref{eqFDGeq1},
    \begin{align*}
        h^0(K_S+D)\geq (f^*K_{S_0}\cdot D)+1 \geq 2.
    \end{align*}

    Take a resolution $g:T\to S$ such that $\Mov|g^*(K_S+D)|$ is base point free. Let $M\in\Mov |g^*(K_S+D)|$.

    Assume that $|K_S+D|$ does not define a generically finite map. Then $|M|$ is composed with a pencil. Since $q(S)=0$, the pencil must be a rational pencil. Let $F$ be a generic irreducible element of $|M|$. Then $M\sim aF$ where
    \begin{align}\label{eqAGeqH0KSD1}
        a\geq h^0(K_S+D)-1.
    \end{align}

    Let $h=f\circ g$, $F_0=h_*F$.

    If $(K_{S_0}\cdot F_0)=0$, then by Hodge index theorem, we have $F_0^2<0$ which is impossible since $F_0$ moves. Thus $(K_{S_0}\cdot F_0)\geq 1$.

    If $(K_{S_0}\cdot F_0)=1$, then since $(K_{S_0}\cdot F_0)+F_0^2$ is even and $F_0^2\geq 0$, we have that $F_0^2\geq 1$. By Hodge index theorem,
    \[
        1=(K_{S_0}\cdot F_0)^2\geq K_{S_0}^2 F_0^2=F_0^2\geq 1.
    \]
    Thus $K_{S_0}\equiv F_0$. Since $q(S_0)=0$, we have that $\OO_{S_0}(K_{S_0}-F_0)$ is a torsion line bundle. By \cref{thmXG2}, we have $h^1(K_{S_0}-F_0)=0$. By Riemann-Roch formula and Serre duality theorem,
    \begin{align*}
        h^0(F_0)-h^1(K_{S_0}-F_0)+h^0(K_{S_0}-F_0)&=\frac{1}{2}(F_0\cdot(F_0-K_{S_0}))+1=1.
    \end{align*}
    Thus $h^0(F_0)\leq h^0(F_0)+h^0(K_{S_0}-F_0)= 1$ which is impossible since $F_0$ moves.

    Thus $(K_{S_0}\cdot F_0)\geq 2$. Thus
    \begin{align*}
        (f^*K_{S_0}\cdot (K_S+D))&=(h^*K_{S_0}\cdot g^*(K_S+D))\\
        &\geq (h^*K_{S_0}\cdot M)=a (h^*K_{S_0}\cdot F)=a (K_{S_0}\cdot F_0)\geq 2a.
    \end{align*}

    With \cref{eqAGeqH0KSD1} and \cref{eqH0KSDGeq},
    \begin{align*}
        (f^*K_{S_0}\cdot (K_S+D))\geq 2a\geq 2h^0(K_S+D)-2
        \geq ((f^*K_{S_0}+D)\cdot D).
    \end{align*}

    Thus $1=(f^*K_{S_0}\cdot K_S)\geq D^2$. By \cref{eqD2GeqFD} and \cref{eqFDGeq1}, $D^2\geq (f^*K_{S_0}\cdot D) \geq 1$. Thus $D^2=(f^*K_{S_0}\cdot D)=1$.

    By Hodge index theorem,
    \[
        1=(f^*K_{S_0}\cdot D)^2\geq (f^*K_{S_0})^2 D^2=1.
    \]
    Thus $f^*K_{S_0}\equiv D$. Since $f^*K_{S_0}\leq D$, we have that $f^*K_{S_0}=D$ which is impossible since $p_g(S)=0$ and $D\geq 0$.

    Thus $|K_S+f^*K_{S_0}+B|=|K_S+D|$ defines a generically finite map.
\end{proof}

\section{An effective comparison inequality}\label{secCompare}

With the help of \cref{propCompareBirG}, we can establish an effective comparison inequality with the stronger relation $\geq$ rather than $\geq_\bQ$.

\begin{proposition}\label{compareIneq}
    Let $X$ be a minimal 3-fold of general type with $P_m(X)\geq 2$. Let $\Lambda$ be a 1-dimensional subsystem of $|mK_X|$. Let $f:W\to X$ be a Chen resolution with respect to $\Lambda$. Let $F$ be a generic irreducible element of $f^*\Lambda$. Let $\sigma$ be the morphism from $F$ to its minimal model $F_0$. Let $D$ be any divisor linearly equivalent to $(m+1)K_X$. Then there exists a divisor $D_0$ linearly equivalent to $K_{F_0}$ such that
    \[
        f^*(D)|_F - \sigma^*D_0
    \]
    is an effective $\bQ$-divisor. We write $f^*(D)|_F \geq \sigma^*D_0$ for brevity.

    Note that we distinguish $\geq$ from $\geq_\bQ$ in this paper.
\end{proposition}

\begin{proof}
    Since $\Lambda$ is 1-dimensional, we have that $\Lambda$ is composed with a pencil. Keep the notation in \cref{propCompareBirG}. By definition of Chen resolution, $f=g\circ h$ for some birational morphism $h$.

    Let $F'$ be a generic irreducible element of $f^*\Lambda$ different from $F$. Let $F'_{X'}=h_*F'$, $F_{X'}=h_*F$, $F'_X=f_*F'$. Let $\sigma'$ be the morphism from $F_{X'}$ to $F_0$. Then $\sigma=\sigma'\circ h|_{F}$.

    By \cref{propCompareBirG}, we have
    \begin{align*}
        g^*(K_X+F'_X)\geq K_{X'}+F'_{X'}
    \end{align*}
    and
    \[
        F'_{X'}|_{F_{X'}}=0.
    \]

    Since $F'$ is a generic irreducible element of $f^*\Lambda$, there exists $L\in\Lambda$ such that $L\geq F'_X$. Thus
    \[
        g^*(K_X+L)\geq K_{X'}+F'_{X'}.
    \]
    Thus
    \[
        g^*(K_X+L)|_{F_{X'}}\geq (K_{X'}+F'_{X'})|_{F_{X'}}=K_{X'}|_{F_{X'}}=K_{F_{X'}}.
    \]

    Let $K'=K_{F_{X'}}+g^*(D-K_X-L)|_{F_{X'}}$. Then
    \[
        g^*(D)|_{F_{X'}}\geq K'.
    \]
    
    Since $D\sim (m+1)K_X$ and $L\sim mK_X$, we have $K'\sim K_{F_{X'}}$.

    Let $D_0=\sigma'_*K'$. Then $D_0\sim K_{F_0}$ and
    \[
        K'-\sigma'^*D_0=K_{F_{X'}}-\sigma'^*K_{F_0}\geq 0.
    \]
    Thus
    \[
        g^*(D)|_{F_{X'}}\geq K'\geq \sigma'^*D_0.
    \]
    Thus
    \[
        f^*(D)|_F = (h|_F)^*(g^*(D)|_{F_{X'}}) \geq (h|_F)^*(\sigma'^*D_0) = \sigma^*D_0.
    \]
\end{proof}

\section{Proof of \texorpdfstring{\cref{thm6K}}{Theorem 1.1}}\label{secThm6K}

\begin{proposition}\label{propH2OX}
    Let $f:X\to C$ be a contraction morphism from a smooth projective 3-fold to a curve. Let $S$ be a general fiber of $f$. If $q(S)=0$, then $h^2(\OX)\leq p_g(S)$.
\end{proposition}

\begin{proof}
    Since $f$ is a contraction morphism, we have that $C$ is normal. Thus $C$ is smooth.

    Let $B=f_*\omega_{X/C}$. By Serre duality theorem,
    \begin{align*}
        h^1(f_*\omega_X)=h^1(B\otimes \omega_C)=h^0(B^\vee).
    \end{align*}
    By Fujita's semipositivity theorem (\cite{Fujita1978}), $B$ is nef on $C$. Thus
    \[
        h^0(B^\vee)\leq \rank B=p_g(S).
    \]

    Since $q(S)=0$, we have $\rank R^1f_*\omega_X=0$. By Koll\'ar's torsion-freeness theorem (\cite{Kollar1986}), we have $R^1f_*\omega_X=0$.

    By Serre duality theorem and Leray spectral sequence, we have
    \begin{align*}
        h^2(\OX)=h^1(\omega_X)&=h^1(f_*\omega_X)+h^0(R^1f_*\omega_X)\\
        &=h^1(f_*\omega_X)\leq p_g(S).
    \end{align*}
\end{proof}

\begin{proposition}\label{propM2}
    Let $X$ be a minimal 3-fold of general type with $\chi(\OX)\neq 1$ or $p_g(X)\neq 0$. If $P_m(X)\geq 2$ and $P_{m+2}(X)>0$, then $|(2m+2)K_X|$ defines a generically finite map.
\end{proposition}

\begin{proof}
    Let $f:W\to X$ be a Chen resolution with respect to a 1-dimensional subsystem $\Lambda\subseteq|mK_X|$. Let $S$ be a generic irreducible element of $f^*\Lambda$. Let $\sigma$ be the morphism from $S$ to its minimal model $S_0$.

    By \cref{coroTankeev1}, since $P_{m+2}(X)>0$, we only need to prove that $|(2m+2)K_W||_S$ defines a generically finite map.
    
    We have
    \[
        |(2m+2)K_W|\lsgeq |K_W+\lceil (m+1)f^*K_X\rceil+S|.
    \]

    By Kawamata–Viehweg vanishing theorem, since $(m+1)f^*K_X$ is nef and big and
    \[
        \Supp(\lceil (m+1)f^*K_X\rceil-(m+1)f^*K_X)
    \]
    is contained in $\Exc(f)$ which is a simple normal crossing divisor, we have
    \[
        h^1(K_W+\lceil (m+1)f^*K_X\rceil)=0.
    \]

    Thus
    \begin{align*}
        |K_W+\lceil (m+1)f^*K_X\rceil+S||_S&=|(K_W+\lceil (m+1)f^*K_X\rceil+S)|_S|\\
        &=|K_S+\lceil (m+1)f^*K_X\rceil|_S|\\
        &\lsgeq |K_S+\lceil (m+1)f^*(K_X)|_S\rceil|.
    \end{align*}

    Thus we only need to prove that $|K_S+\lceil (m+1)f^*(K_X)|_S\rceil|$ defines a generically finite map.

    By \cref{compareIneq},
    \[
        |K_S+\lceil (m+1)f^*(K_X)|_S\rceil|\lsgeq |K_S+\sigma^*(K_{S_0})|\lsgeq \sigma^*|2K_{S_0}|.
    \]
    
    Thus we only need to prove that $|2K_{S_0}|$ defines a generically finite map.

    By \cref{thmXG1}, we only need to consider the case that $S_0$ is a $(1,0)$-surface. Thus $q(S)=p_g(S)=0$. By \cref{propH2OX}, $h^2(\OO_W)=0$.

    Since $p_g(S)=0$, we have that the natural homomorphism
    \[
        H^0(K_W-S)\to H^0(K_W)
    \]
    is surjective. Thus $p_g(X)=h^0(K_W)=0$.

    Since $\chi(\OX)\neq 1$ or $p_g(X)\neq 0$, we have $\chi(\OX)\neq 1$. Since $h^2(\OX)=h^2(\OO_W)=0$, we have $q(X)=1-\chi(\OX)\neq 0$. Since $P_m(X)\geq 2$ and $p_g(X)=0$, we have $m\geq 2$. Thus $2m+2\geq 5$. By \cref{thmIrr5}, $|(2m+2)K_X|$ defines a birational map.
\end{proof}

\begin{proposition}\label{propMN221}
    Let $X$ be a minimal 3-fold of general type. If $P_m(X)\geq 2$, $P_{m+1}(X)>0$ and $P_{m+2}(X)>0$, then $|(2m+2)K_X|$ defines a generically finite map.
\end{proposition}

\begin{proof}
    Let $f:W\to X$ be a Chen resolution with respect to a 1-dimensional subsystem $\Lambda\subseteq|mK_X|$. Let $S$ be a generic irreducible element of $f^*\Lambda$. Let $\sigma$ be the morphism from $S$ to its minimal model $S_0$.

    By \cref{coroTankeev1}, since $P_{m+2}(X)>0$, we only need to prove that $|(2m+2)K_W||_S$ defines a generically finite map.

    We have
    \[
        |(2m+2)K_W|\lsgeq |K_W+\lceil (m+1)f^*K_X \rceil +S|.
    \]

    By Kawamata–Viehweg vanishing theorem, we only need to prove that
    \[
        |K_S+\lceil (m+1)f^*(K_X) |_S\rceil|
    \]
    defines a generically finite map.

    Since $P_{m+1}(X)\geq 1$, there exists a divisor $D\in |(m+1)K_X|$. By \cref{compareIneq}, there exists a divisor $D_0\sim K_{S_0}$ such that
    \[
        f^*(D)|_S \geq \sigma^*D_0.
    \]

    Since $D\geq 0$, we have $f^*(D)|_S\geq 0$. Let $B=\lceil f^*(D)|_S \rceil-\sigma^*D_0$. Then $B\geq 0$ and $\sigma^*D_0+B\geq 0$.

    Thus
    \[
        |K_S+\lceil (m+1)f^*(K_X) |_S\rceil|= |K_S+\sigma^*D_0+B|
    \]
    defines a generically finite map by \cref{thmMain1}.
\end{proof}

\begin{proposition}\label{propPmPm1}
    Let $X$ be a minimal 3-fold of general type. If $P_m(X)\geq 2\chi(\OX)$ for some $m\geq 2$, then $P_{m+1}(X)>0$.
\end{proposition}

\begin{proof}
    By Reid's Riemann-Roch formula and Kawamata–Viehweg vanishing theorem, for any $m\geq 2$, we have
    \[
        P_m(X)=\frac{1}{12}m(m-1)(2m-1)K_X^3-(2m-1)\chi(\OX)+l(m).
    \]

    Thus, for any $m\geq 2$, we have
    \begin{align*}
        P_{m+1}(X)-P_m(X)&=\frac{1}{2}m^2K_X^3-2\chi(\OX)+l(m+1)-l(m)\\
        &\geq \frac{1}{2}m^2K_X^3-2\chi(\OX)> -2\chi(\OX).
    \end{align*}
\end{proof}

\begin{proposition}\label{prop2m2}
    Let $X$ be a minimal 3-fold of general type with $P_m(X)\geq 2$ and $P_{m+2}(X)>0$. Then $|(2m+2)K_X|$ defines a generically finite map.
\end{proposition}

\begin{proof}
    \textbf{Case 1}. $\chi(\OX)\neq 1$ or $m=1$.

    By \cref{propM2}, $|(2m+2)K_X|$ defines a generically finite map.

    \medskip

    \textbf{Case 2}. $\chi(\OX)= 1$ and $m\geq 2$.

    By \cref{propPmPm1}, $P_{m+1}(X)>0$. By \cref{propMN221}, $|(2m+2)K_X|$ defines a generically finite map.
\end{proof}

\begin{proof}[Proof of \cref{thm6K}]
    Take $m=2$ in \cref{prop2m2}.
\end{proof}

\section{A non-vanishing theorem for threefolds}\label{secNonV}

In this section, we establish an effective non-vanishing theorem for 3-folds and improve the bound 27 in \cite[Theorem 1.1 (1)]{EXP2}.

\begin{theorem}\label{thmNV}
    Let $X$ be a minimal 3-fold of general type with $P_m(X)\geq 2$. Then $P_n(X)>0$ when $n\geq 2m$.
\end{theorem}

\begin{proof}
    Since $P_m(X)>0$, we have $P_{2m}(X)\geq P_m(X)\geq 2$. Thus we may assume $n\geq 2m+1$.

    We may assume $p_g(X)=0$. Then $m\geq 2$. Then $n\geq 2m+1\geq 5$. By \cref{thmIrr5}, we may assume $q(X)=0$.

    Let $f:W\to X$ be a Chen resolution with respect to a 1-dimensional subsystem $\Lambda\subseteq|mK_X|$. Let $S$ be a generic irreducible element of $f^*\Lambda$. Let $\sigma$ be the morphism from $S$ to its minimal model $S_0$. Since $q(X)=0$, the pencil here must be a rational pencil.

    \medskip

    \textbf{Case 1}. $p_g(S)=0$.

    Since $S$ is of general type, we have $q(S)=0$. By \cref{propH2OX}, $h^2(\OO_W)=0$. Thus $\chi(\OX)=1+h^2(\OX)=1$.

    \medskip

    \textbf{Case 1-1}. $m\geq 3$.

    By \cref{thmPChi1}, since $n\geq 2m+1\geq 7$, we have $P_n(X)>0$.

    \medskip

    \textbf{Case 1-2}. $m= 2$.

    By \cref{propPChi1}, since $n-2\geq 2$, we have $P_n(X)\geq P_2(X)\geq 2$.

    \medskip

    \textbf{Case 2}. $p_g(S)>0$ and $n\geq 2m+3$.

    By \cite[Proposition 2.15 (i)]{EXP2}, $P_n(X)\geq 2$.

    \medskip

    \textbf{Case 3}. $p_g(S)>0$ and $2m+1\leq n\leq 2m+2$.

    We have
    \[
        |nK_W|\lsgeq |K_W+\lceil (n-m-1)f^*K_X\rceil+S|.
    \]

    With Kawamata–Viehweg vanishing theorem,
    \[
        |nK_W||_S\lsgeq |K_S+\lceil (n-m-1)f^*(K_X)|_S\rceil|.
    \]

    Thus we only need to prove that
    \[
        h^0(K_S+\lceil (n-m-1)f^*(K_X)|_S\rceil)>0.
    \]

    \medskip

    \textbf{Case 3-1}. $n= 2m+2$.

    By \cref{compareIneq}, we only need to prove that $h^0(K_S+ \sigma^*K_{S_0})>0$ which is true since
    \[
        h^0(K_S+ \sigma^*K_{S_0})=h^0(2K_{S_0})\geq 2.
    \]

    \medskip

    \textbf{Case 3-2}. $n= 2m+1$.

    Since $P_m(X)>0$, we have $h^0(\lceil mf^*(K_X)|_S\rceil)>0$. Thus
    \[
        h^0(K_S+\lceil mf^*(K_X)|_S\rceil)\geq h^0(K_S)>0.
    \]
\end{proof}

\begin{corollary}
    Let $X$ be a minimal 3-fold of general type. Then $P_n(X)>0$ when $n\geq 24$.
\end{corollary}

\begin{proof}
    By \cref{coroDelta}, we may assume $\delta(X)\leq 12$.

    By \cref{thmNV}, since $n\geq 24\geq 2\delta(X)$, we have $P_n(X)>0$.
\end{proof}

\section{Generic finiteness criteria for threefolds}\label{secCriteria3}

In this section, we establish several generic finiteness criteria for 3-folds based on the relationship between $K^3$ and $P_n$.

At first, we prove some basic criteria for surfaces that will be used.

\begin{proposition}\label{propSurfL2B}
    Let $S$ be a smooth projective surface of general type. Let $L$ be a nef and big $\bQ$-divisor on $S$. Let $B$ be a divisor on $S$ such that $h^0(B)\geq 2$. Then
    \[
        |K_S+\lceil L\rceil+2B|
    \]
    defines a generically finite map.
\end{proposition}

\begin{proof}
    Let $g:T\to S$ be a Chen resolution with respect to a 1-dimensional subsystem $V\subseteq|B|$. Let $C$ be a generic irreducible element of $g^*V$.

    Since
    \begin{align*}
        |K_T+\lceil g^*L \rceil+2C| &\lsleq|K_T+g^*\lceil L \rceil+2C|\\
        &\lsleq|K_T+g^*(\lceil L \rceil + 2B)|\\
        &=K_T-g^*K_S+g^*|K_S+\lceil L \rceil + 2B|,
    \end{align*}
    we only need to prove that $|K_T+\lceil g^*L \rceil+2C|$ defines a generically finite map.

    By Kawamata–Viehweg vanishing theorem for surfaces which does not need the simple normal crossing condition,
    \[
        h^1(K_T+\lceil g^*L \rceil)=0.
    \]

    By \cref{coroTankeev2}, we only need to prove that
    \[
        |K_T+\lceil g^*L \rceil+2C||_C
    \]
    defines a generically finite map.

    By Kawamata–Viehweg vanishing theorem, we only need to prove that $|K_C+\lceil g^*(L)|_C \rceil|$ defines a generically finite map.

    We have $\deg_C(\lceil g^*(L)|_C \rceil)\geq (g^*(L)\cdot C)>0$ since $C$ moves. We have $g(C)\geq 2$ since $C$ is a general fiber. Thus
    \[
        h^0(K_C+\lceil g^*(L)|_C \rceil)\geq 2.
    \]
    Thus $|K_C+\lceil g^*(L)|_C \rceil|$ defines a generically finite map.
\end{proof}

\begin{proposition}\label{propSurfLB}
    Let $S$ be a smooth projective surface of general type. Let $L$ be a nef and big $\bQ$-divisor on $S$. Let $V$ be a 1-dimensional linear system on $S$. Let $B\in V$. Let $\Lambda$ be a linear system on $S$ such that $\Lambda\lsgeq V$ and
    \[
        \Lambda\lsgeq |K_S+\lceil L\rceil+B|.
    \]
    Then $\Lambda$ defines a generically finite map.
\end{proposition}

\begin{proof}
    Let $g:T\to S$ be a Chen resolution with respect to $V$. Let $C$ be a generic irreducible element of $g^*V$.

    By \cref{thmTankeev}, since $g^*\Lambda\lsgeq g^*V$, we only need to prove that $g^*(\Lambda)|_C$ defines a generically finite map.

    Since $g^*\Lambda\lsgeq g^*|K_S+\lceil L\rceil+B|$ and
    \begin{align*}
        |K_T+\lceil g^*L \rceil+C| &\lsleq|K_T+g^*\lceil L \rceil+C|\\
        &\lsleq|K_T+g^*(\lceil L \rceil + B)|\\
        &=K_T-g^*K_S+g^*|K_S+\lceil L \rceil + B|,
    \end{align*}
    we only need to prove that $|K_T+\lceil g^*L \rceil+C||_C$ defines a generically finite map.

    By Kawamata–Viehweg vanishing theorem, we only need to prove that $|K_C+\lceil g^*(L)|_C \rceil|$ defines a generically finite map.

    We have $\deg_C(\lceil g^*(L)|_C \rceil)\geq (g^*(L)\cdot C)>0$ since $C$ moves. We have $g(C)\geq 2$ since $C$ is a general fiber. Thus
    \[
        h^0(K_C+\lceil g^*(L)|_C \rceil)\geq 2.
    \]
    Thus $|K_C+\lceil g^*(L)|_C \rceil|$ defines a generically finite map.
\end{proof}

\begin{proposition}\label{propSurfPencil}
    Let $S$ be a smooth projective surface of general type on which a linear system $V\subseteq |B|$ is composed with a pencil where $B$ is a divisor on $S$. Let $L$ be a nef and big $\bQ$-divisor on $S$. Let $\Lambda$ be a linear system on $S$ such that $\Lambda\lsgeq V$ and
    \[
        \Lambda\lsgeq |K_S+\lceil L\rceil|.
    \]
    If $b L\geq B$ for some positive rational number $b<\dim V$, then $\Lambda$ defines a generically finite map.
\end{proposition}

\begin{proof}
    Let $g:T\to S$ be a Chen resolution with respect to $V$. Let $C$ be a generic irreducible element of $g^*V$. Let $M\in\Mov g^*V$. Then $M\equiv aC$ where $a\geq \dim V$.

    By \cref{thmTankeev}, since $g^*\Lambda\lsgeq g^*V$, we only need to prove that $g^*(\Lambda)|_C$ defines a generically finite map.

    Since $g^*\Lambda\lsgeq g^*|K_S+\lceil L\rceil|$ and
    \begin{align*}
        |K_T+\lceil g^*L \rceil| &\lsleq|K_T+g^*\lceil L \rceil|\\
        &=K_T-g^*K_S+g^*|K_S+\lceil L \rceil|,
    \end{align*}
    we only need to prove that $|K_T+\lceil g^*L \rceil||_C$ defines a generically finite map.

    Since $M\in\Mov g^*V$, there exists $Z\geq 0$ such that $M+Z\in | g^*B |$.

    Let $D=\frac{1}{a}(b g^*L-g^*B+Z)$. Then $D$ is an effective $\bQ$-divisor since $b L\geq B$. Thus we only need to prove that
    \[
        |K_T+\lceil g^*L-D-C\rceil+C||_C
    \]
    defines a generically finite map.

    Let $Q=g^*L-D-C$. Then $Q\equiv (1-\frac{b}{a})g^*L$. We have that $Q$ is nef and big since $a\geq \dim V>b$.

    By Kawamata–Viehweg vanishing theorem, we only need to prove that $|K_C+\lceil Q|_C\rceil|$ defines a generically finite map.

    We have $\deg_C(\lceil Q|_C \rceil)\geq (Q\cdot C)>0$ since $C$ moves. We have $g(C)\geq 2$ since $C$ is a general fiber. Thus
    \[
        h^0(K_C+\lceil Q|_C \rceil)\geq 2.
    \]
    Thus $|K_C+\lceil Q|_C \rceil|$ defines a generically finite map.
\end{proof}

\begin{proposition}\label{propD2}
    Let $X$ be a minimal 3-fold of general type with $P_m(X)\geq 3$. If $|mK_X|$ is not composed with a pencil, then $|nK_X|$ defines a generically finite map when $P_{n-m}(X)>0$ and
    \[
        n> \frac{m}{P_m(X)-2} +m+1.
    \]
\end{proposition}

\begin{proof}
    Let $f:W\to X$ be a Chen resolution with respect to $\Lambda=|mK_X|$. Let $S$ be a generic irreducible element of $f^*\Lambda$.

    By \cref{coroTankeev1}, since $P_{n-m}(X)>0$, we only need to prove that $|nK_W||_S$ defines a generically finite map.

    We have
    \[
        |nK_W|\lsgeq |K_W+\lceil (n-m-1)f^*K_X \rceil +S|.
    \]

    With Kawamata–Viehweg vanishing theorem,
    \begin{align}\label{eqNKSGeqNM1}
        |nK_W||_S\lsgeq |K_S+\lceil (n-m-1)f^*(K_X) |_S\rceil|.
    \end{align}

    Let $V$ be the linear system on $S$ corresponding to the image of the natural homomorphism
    \[
        H^0(W, \lfloor mf^*K_X \rfloor)\to H^0(S, \lfloor mf^*K_X \rfloor|_S).
    \]

    Since $P_{n-m}(X)>0$, we have $|nK_W|\lsgeq |\lfloor mf^*K_X \rfloor|$. Thus
    \begin{align}\label{eqNKSGeqV}
        |nK_W||_S\lsgeq |\lfloor mf^*K_X \rfloor||_S=V.
    \end{align}

    Thus we may assume that $V$ does not define a generically finite map.

    Since $\Lambda$ is not composed with a pencil, we have $h^0(\lfloor mf^*K_X \rfloor -S)=1$. Thus
    \begin{align*}
        \dim V&=h^0(\lfloor mf^*K_X \rfloor)-h^0(\lfloor mf^*K_X \rfloor -S)-1\\
        &=P_m(X)-2 \geq 1.
    \end{align*}

    Thus $V$ is composed with a pencil.

    Let $L=(n-m-1)f^*(K_X) |_S$, $B=\lfloor mf^*K_X \rfloor|_S$, $b=\frac{m}{n-m-1}$. Then $bL\geq B$ and $b<P_m(X)-2= \dim V$.

    By \cref{propSurfPencil}, \cref{eqNKSGeqNM1} and \cref{eqNKSGeqV}, we have that $|nK_W||_S$ defines a generically finite map.
\end{proof}

\begin{proposition}\label{propMasek3}
    Let $X$ be a minimal 3-fold of general type with $P_m(X)\geq 3$. Then $|nK_X|$ defines a generically finite map when $P_{n-m}(X)>0$ and
    \[
        n>\frac{P_m(X)+2}{P_m(X)-1}m+4.
    \]
\end{proposition}

\begin{proof}
    Let $f:W\to X$ be a Chen resolution with respect to $\Lambda=|mK_X|$. Let $S$ be a generic irreducible element of $f^*\Lambda$. Let $\sigma$ be the morphism from $S$ to its minimal model $S_0$. Let $p=\dim\Lambda$. Then $p=P_m(X)-1$. Let $d=\dim\overline{\Phi_\Lambda(X)}$.

    \medskip

    \textbf{Case 1}. $d=1$.

    By \cref{coroTankeev1}, since $P_{n-m}(X)>0$, we only need to prove that $|nK_W||_S$ defines a generically finite map.

    We have
    \[
        |nK_W|\lsgeq |K_W+\lceil (n-m-1)f^*K_X \rceil +S|.
    \]

    With Kawamata–Viehweg vanishing theorem,
    \begin{align}\label{eqNKSGeqNM1-Masek3}
        |nK_W||_S\lsgeq |K_S+\lceil (n-m-1)f^*(K_X) |_S\rceil|.
    \end{align}

    Let $L=(n-m-1)f^*(K_X) |_S$. By \cref{propQCompare},
    \[
        L\geq_\bQ \frac{(n-m-1)p}{m+p}\sigma^*K_{S_0}.
    \]

    Thus
    \begin{align*}
        L^2\geq \left(\frac{(n-m-1)p}{m+p}\right)^2 K_{S_0}^2 >9,
    \end{align*}
    and
    \begin{align*}
        (L\cdot C)&\geq \frac{(n-m-1)p}{m+p}(\sigma^*K_{S_0}\cdot C)\\
        &\geq \frac{(n-m-1)p}{m+p}>3>\frac{4}{1+\sqrt{1-\frac{8}{L^2}}}
    \end{align*}
    for any curve $C$ on $S$ passing through two very general points of $S$.

    By \cref{eqNKSGeqNM1-Masek3} and \cref{thmMasek}, $|nK_W||_S\lsgeq|K_S+\lceil L\rceil|$ defines a birational map.

    \medskip

    \textbf{Case 2}. $d\geq 2$.

    By \cref{propD2}, since
    \[
        n>\frac{P_m(X)+2}{P_m(X)-1}m+4> \frac{m}{P_m(X)-2} +m+1,
    \]
    we have that $|nK_X|$ defines a generically finite map.
\end{proof}

\begin{proposition}\label{prop2sqrt2}
    Let $X$ be a minimal 3-fold of general type with $P_m(X)\geq 2$. Then $|nK_X|$ defines a generically finite map when
    \[
        n>(2\sqrt{2}+1)(m+1).
    \]
\end{proposition}

\begin{proof}
    Let $f:W\to X$ be a Chen resolution with respect to a 1-dimensional subsystem $\Lambda\subseteq|mK_X|$. Let $S$ be a generic irreducible element of $f^*\Lambda$. Let $\sigma$ be the morphism from $S$ to its minimal model $S_0$.

    By \cref{thmNV}, since $n-m\geq 2m$, we have $P_{n-m}(X)>0$. By \cref{coroTankeev1}, we only need to prove that $|nK_W||_S$ defines a generically finite map.

    We have
    \[
        |nK_W|\lsgeq |K_W+\lceil (n-m-1)f^*K_X \rceil +S|.
    \]

    With Kawamata–Viehweg vanishing theorem,
    \begin{align}\label{eqNKSGeqNM1-2sqrt2}
        |nK_W||_S\lsgeq |K_S+\lceil (n-m-1)f^*(K_X) |_S\rceil|.
    \end{align}

    \medskip

    \textbf{Case 1}. $S$ is not a $(1,2)$-surface.

    Let $L=(n-m-1)f^*(K_X) |_S$. By \cref{propQCompare},
    \[
        L\geq_\bQ \frac{n-m-1}{m+1}\sigma^*K_{S_0}.
    \]

    Thus
    \begin{align*}
        L^2\geq \left(\frac{n-m-1}{m+1}\right)^2 K_{S_0}^2 > 8.
    \end{align*}

    By \cref{propGeneralKC2},
    \begin{align*}
        (L\cdot C)\geq \frac{n-m-1}{m+1}(\sigma^*K_{S_0}\cdot C) \geq \frac{2(n-m-1)}{m+1}>4
    \end{align*}
    for any curve $C$ on $S$ passing through two very general points of $S$.

    By \cref{eqNKSGeqNM1-2sqrt2} and \cref{thmMasek}, $|nK_W||_S\lsgeq|K_S+\lceil L\rceil|$ defines a birational map.

    \medskip

    \textbf{Case 2}. $S$ is a $(1,2)$-surface.

    By \cref{eqNKSGeqNM1-2sqrt2} and \cref{compareIneq}, we only need to prove that
    \[
        |K_S+\lceil (n-3m-3)f^*(K_X) |_S\rceil+2\sigma^*K_{S_0}|
    \]
    defines a generically finite map which is true by \cref{propSurfL2B} and $p_g(S)=2$.
\end{proof}

\begin{proposition}\label{prop3m2}
    Let $X$ be a minimal 3-fold of general type with $P_m(X)\geq 2$. Then
    \[
        |(lm+l+2m+1)K_X|
    \]
    defines a generically finite map for every integer $l\geq 1$.
\end{proposition}

\begin{proof}
    Let $f:W\to X$ be a Chen resolution with respect to a 1-dimensional subsystem $\Lambda\subseteq|mK_X|$. Let $S$ be a generic irreducible element of $f^*\Lambda$. Let $\sigma$ be the morphism from $S$ to its minimal model $S_0$.

    We have
    \[
        |(lm+l+2m+1)K_W|\lsgeq |K_W+\lceil l(m+1)f^*K_X \rceil+2S|.
    \]

    By Kawamata–Viehweg vanishing theorem,
    \[
        h^1(K_W+\lceil l(m+1)f^*K_X \rceil)=0.
    \]

    By \cref{coroTankeev2}, we only need to prove that
    \[
        |K_W+\lceil l(m+1)f^*K_X \rceil+2S||_S
    \]
    defines a generically finite map.

    By Kawamata–Viehweg vanishing theorem, we only need to prove that
    \[
        |K_S+\lceil l(m+1)f^*(K_X)|_S \rceil|
    \]
    defines a generically finite map.

    \medskip

    \textbf{Case 1}. $l\geq 2$ or $S$ is not a $(1,0)$-surface.

    By \cref{compareIneq},
    \[
        |K_S+\lceil l(m+1)f^*(K_X)|_S \rceil|\lsgeq |K_S+l\sigma^*K_{S_0}|\lsgeq \sigma^*|(l+1)K_{S_0}|.
    \]

    Thus we only need to prove that $|(l+1)K_{S_0}|$ defines a generically finite map which is true by \cref{thmXG1} and \cref{thmSurf3K}.

    \medskip

    \textbf{Case 2}. $l=1$ and $S$ is a $(1,0)$-surface.

    Since $q(S)=p_g(S)=0$, by \cref{propH2OX}, $h^2(\OO_W)=0$.

    Since $p_g(S)=0$, we have that the natural homomorphism
    \[
        H^0(K_W-S)\to H^0(K_W)
    \]
    is surjective. Thus $p_g(X)=h^0(K_W)=0$.

    By \cref{thmIrr5}, we may assume $q(X)=0$. Then $\chi(\OX)=1$.

    Since $p_g(X)=0$, we have $m\geq 2$. By \cref{propPmPm1}, $P_{m+1}(X)>0$. Take a divisor $D\in |(m+1)K_X|$. By \cref{compareIneq}, there exists a divisor $D_0\sim K_{S_0}$ such that
    \[
        f^*(D)|_S \geq \sigma^*D_0.
    \]

    Since $D\geq 0$, we have $f^*(D)|_S\geq 0$. Let $B=\lceil f^*(D)|_S \rceil-\sigma^*D_0$. Then $B\geq 0$ and $\sigma^*D_0+B\geq 0$.

    Thus
    \[
        |K_S+\lceil (m+1)f^*(K_X) |_S\rceil|= |K_S+\sigma^*D_0+B|
    \]
    defines a generically finite map by \cref{thmMain1}.
\end{proof}

\begin{proposition}\label{prop3m3}
    Let $X$ be a minimal 3-fold of general type with $P_m(X)\geq 2$. Then
    \[
        |(lm+l)K_X|
    \]
    defines a generically finite map for every integer $l\geq 3$.
\end{proposition}

\begin{proof}
    Let $f:W\to X$ be a Chen resolution with respect to a 1-dimensional subsystem $\Lambda\subseteq|mK_X|$. Let $S$ be a generic irreducible element of $f^*\Lambda$. Let $\sigma$ be the morphism from $S$ to its minimal model $S_0$.

    By \cref{thmNV}, since $lm+l-m\geq 2m$, we have $P_{lm+l-m}(X)>0$. By \cref{coroTankeev1}, we only need to prove that $|(lm+l)K_W||_S$ defines a generically finite map.

    We have
    \[
        |(lm+l)K_W|\lsgeq |K_W+\lceil (l-1)(m+1)f^*K_X \rceil +S|.
    \]

    With Kawamata–Viehweg vanishing theorem,
    \begin{align*}
        |(lm+l)K_W||_S\lsgeq |K_S+\lceil (l-1)(m+1)f^*(K_X) |_S\rceil|.
    \end{align*}

    With \cref{compareIneq},
    \[
        |(lm+l)K_W||_S\lsgeq |K_S+ (l-1)\sigma^*K_{S_0}|\lsgeq \sigma^*|lK_{S_0}|.
    \]

    Thus we only need to prove that $|lK_{S_0}|$ defines a generically finite map which is true by \cref{thmSurf3K}.
\end{proof}

\begin{proposition}\label{propMasekChi3}
    Let $X$ be a minimal 3-fold of general type with $\chi(\OX)\geq 3$, $q(X)=p_g(X)=0$ and $P_m(X)\geq 2$. Then $|nK_X|$ defines a generically finite map when $n\geq 3m+2$.
\end{proposition}

\begin{proof}
    \textbf{Case 1}. $n= 3m+2$.

    Take $l=1$ in \cref{prop3m2}.

    \medskip

    \textbf{Case 2}. $n= 3m+3$.

    Take $l=3$ in \cref{prop3m3}.

    \medskip

    \textbf{Case 3}. $n\geq 3m+4$.

    Let $f:W\to X$ be a Chen resolution with respect to a 1-dimensional subsystem $\Lambda\subseteq|mK_X|$. Let $S$ be a generic irreducible element of $f^*\Lambda$. Let $\sigma$ be the morphism from $S$ to its minimal model $S_0$.

    By \cref{thmNV}, since $n-m\geq 2m$, we have $P_{n-m}(X)>0$. By \cref{coroTankeev1}, we only need to prove that $|nK_W||_S$ defines a generically finite map.

    We have
    \[
        |nK_W|\lsgeq |K_W+\lceil (n-m-1)f^*K_X \rceil +S|.
    \]

    With Kawamata–Viehweg vanishing theorem,
    \begin{align}\label{eqNKSGeqNM1-MasekChi3}
        |nK_W||_S\lsgeq |K_S+\lceil (n-m-1)f^*(K_X) |_S\rceil|.
    \end{align}

    Since $\chi(\OX)\geq 3$ and $q(X)=p_g(X)=0$, we have $h^2(\OX)\geq 2$. By \cref{propH2OX}, we have that $S$ is neither a $(1,0)$-surface nor a $(1,1)$-surface. Thus either $K_{S_0}^2\geq 2$ or $S$ is a $(1,2)$-surface.

    \medskip

    \textbf{Case 3-1}. $K_{S_0}^2\geq 2$.

    Let $L=(n-m-1)f^*(K_X) |_S$. By \cref{propQCompare},
    \[
        L\geq_\bQ \frac{n-m-1}{m+1}\sigma^*K_{S_0}.
    \]

    Thus
    \begin{align*}
        L^2\geq \left(\frac{n-m-1}{m+1}\right)^2 K_{S_0}^2\geq 2\left(\frac{n-m-1}{m+1}\right)^2 > 8.
    \end{align*}

    By \cref{propGeneralKC2},
    \begin{align*}
        (L\cdot C)\geq \frac{n-m-1}{m+1}(\sigma^*K_{S_0}\cdot C) \geq \frac{2(n-m-1)}{m+1}>4
    \end{align*}
    for any curve $C$ on $S$ passing through two very general points of $S$.

    By \cref{eqNKSGeqNM1-MasekChi3} and \cref{thmMasek}, $|nK_W||_S\lsgeq|K_S+\lceil L\rceil|$ defines a birational map.

    \medskip

    \textbf{Case 3-2}. $S$ is a $(1,2)$-surface.

    By \cref{eqNKSGeqNM1-MasekChi3} and \cref{compareIneq}, we only need to prove that
    \[
        |K_S+\lceil (n-3m-3)f^*(K_X) |_S\rceil+2\sigma^*K_{S_0}|
    \]
    defines a generically finite map which is true by \cref{propSurfL2B} and $p_g(S)=2$.
\end{proof}

\begin{proposition}\label{propDeltaM2N3}
    Let $X$ be a minimal 3-fold of general type with $\chi(\OX)\neq 1$, $q(X)=p_g(X)=0$, $P_m(X)\geq 2$, $P_n(X)\geq 3$, $P_{n-m}(X)>0$ and $n \geq 2m+3$. Let $p=\lfloor \frac{P_n(X)-1}{2}\rfloor$. If
    \[
        K_X^3<\frac{4p^3}{n(n+p)(3n+4p)},
    \]
    then $|nK_X|$ defines a generically finite map.
\end{proposition}

\begin{proof}
    Let $f:W\to X$ be a Chen resolution with respect to a 1-dimensional subsystem $\Lambda\subseteq|mK_X|$. Let $S$ be a generic irreducible element of $f^*\Lambda$. Since $q(X)=0$, the pencil here must be a rational pencil.

    By \cref{coroTankeev1}, since $P_{n-m}(X)>0$, we only need to prove that $|nK_W||_S$ defines a generically finite map.

    We have
    \[
        |nK_W|\lsgeq |K_W+\lceil (n-m-1)f^*K_X \rceil +S|.
    \]

    With Kawamata–Viehweg vanishing theorem,
    \begin{align}\label{eqNKSGeqNM1-DeltaM2N3}
        |nK_W||_S\lsgeq |K_S+\lceil (n-m-1)f^*(K_X) |_S\rceil|.
    \end{align}

    Since $\chi(\OX)\neq 1$ and $q(X)=p_g(X)=0$, we have $h^2(\OX)>0$. By \cref{propH2OX}, we have that $S$ is not a $(1,0)$-surface.

    Let $V$ be the linear system on $S$ corresponding to the image of the natural homomorphism
    \[
        H^0(W, \lfloor nf^*K_X \rfloor)\to H^0(S, \lfloor nf^*K_X \rfloor|_S).
    \]

    Since $|nK_W|\lsgeq |\lfloor nf^*K_X \rfloor|$, we have
    \begin{align}\label{eqNKSGeqV-DeltaM2N3}
        |nK_W||_S\lsgeq |\lfloor nf^*K_X \rfloor||_S=V.
    \end{align}

    Thus we may assume that $V$ does not define a generically finite map. Then either $\dim V=0$ or $V$ is composed with a pencil.

    \medskip

    \textbf{Case 1}. $\dim V\geq 2$.

    Let $L=(n-m-1)f^*(K_X)|_S$, $B=\lfloor nf^*K_X \rfloor|_S$, $b=\frac{n}{n-m-1}$. Then $bL\geq B$ and $b<2\leq\dim V$.

    By \cref{propSurfPencil}, \cref{eqNKSGeqNM1-DeltaM2N3} and \cref{eqNKSGeqV-DeltaM2N3}, we have that $|nK_W||_S$ defines a generically finite map.

    \medskip

    \textbf{Case 2}. $\dim V\leq 1$.

    Since
    \begin{align*}
        P_n(X)=h^0(\lfloor nf^*K_X \rfloor)\leq h^0(\lfloor nf^*K_X \rfloor-pS)+p(\dim V+1),
    \end{align*}
    we have
    \begin{align*}
        h^0(\lfloor nf^*K_X \rfloor-pS)\geq P_n(X)-2p\geq 1.
    \end{align*}

    Let $\widetilde\Lambda=|pS|$. Then $\widetilde\Lambda\lsleq |\lfloor nf^*K_X \rfloor|$, $\dim \widetilde\Lambda=p$ and $\widetilde\Lambda$ is composed with a pencil.

    By \cref{propK3Pencil},
    \[
        K_X^3\geq \frac{4p^3}{n(n+p)(3n+4p)}
    \]
    which is a contradiction.
\end{proof}

\begin{proposition}\label{propDeltaM2L3}
    Let $X$ be a minimal 3-fold of general type with $\chi(\OX)\neq 1$, $q(X)=p_g(X)=0$, $P_m(X)\geq 2$, $P_l(X)\geq 3$, $P_{n-m}(X)>0$, $P_{n-l}(X)>0$, $n\geq l+2$ and $2n\geq l+2m+3$. Let $p=P_l(X)-2$. If
    \[
        K_X^3<\frac{4p^3}{l(l+p)(3l+4p)},
    \]
    then $|nK_X|$ defines a generically finite map.
\end{proposition}

\begin{proof}
    Let $f:W\to X$ be a Chen resolution with respect to a 1-dimensional subsystem $\Lambda\subseteq|mK_X|$. Let $S$ be a generic irreducible element of $f^*\Lambda$. Since $q(X)=0$, the pencil here must be a rational pencil.

    By \cref{coroTankeev1}, since $P_{n-m}(X)>0$, we only need to prove that $|nK_W||_S$ defines a generically finite map.

    We have
    \[
        |nK_W|\lsgeq |K_W+\lceil (n-m-1)f^*K_X \rceil +S|.
    \]

    With Kawamata–Viehweg vanishing theorem,
    \begin{align}\label{eqNKSGeqNM1-DeltaM2L3}
        |nK_W||_S\lsgeq |K_S+\lceil (n-m-1)f^*(K_X) |_S\rceil|.
    \end{align}

    Since $\chi(\OX)\neq 1$ and $q(X)=p_g(X)=0$, we have $h^2(\OX)>0$. By \cref{propH2OX}, we have that $S$ is not a $(1,0)$-surface.

    Let $V$ be the linear system on $S$ corresponding to the image of the natural homomorphism
    \[
        H^0(W, \lfloor lf^*K_X \rfloor)\to H^0(S, \lfloor lf^*K_X \rfloor|_S).
    \]

    Since $P_{n-l}(X)>0$, we have $|nK_W|\lsgeq |\lfloor lf^*K_X \rfloor|$. Thus
    \begin{align}\label{eqNKSGeqV-DeltaM2L3}
        |nK_W||_S\lsgeq |\lfloor lf^*K_X \rfloor||_S=V.
    \end{align}

    Thus we may assume that $V$ does not define a generically finite map. Then either $\dim V=0$ or $V$ is composed with a pencil.

    \medskip

    \textbf{Case 1}. $\dim V\geq 2$.

    Let $L=(n-m-1)f^*(K_X)|_S$, $B=\lfloor lf^*K_X \rfloor|_S$, $b=\frac{l}{n-m-1}$. Then $bL\geq B$ and $b<2\leq\dim V$.

    By \cref{propSurfPencil}, \cref{eqNKSGeqNM1-DeltaM2L3} and \cref{eqNKSGeqV-DeltaM2L3}, we have that $|nK_W||_S$ defines a generically finite map.

    \medskip

    \textbf{Case 2}. $\dim V\leq 1$.

    In this case, we have
    \begin{align*}
        h^0(\lfloor lf^*K_X \rfloor-S)&=h^0(\lfloor lf^*K_X \rfloor)-\dim V-1\\
        &=P_l(X)-\dim V-1\geq P_l(X)-2=p\geq 1.
    \end{align*}

    \medskip

    \textbf{Case 2-1}. $p=1$ or $|\lfloor lf^*K_X \rfloor-S|$ and $f^*\Lambda$ are composed with the same pencil.

    In this case, we have $|\lfloor lf^*K_X \rfloor-S|\lsgeq |(p-1)S|$.

    Let $\widetilde\Lambda=|pS|$. Then $\widetilde\Lambda\lsleq |\lfloor lf^*K_X \rfloor|$, $\dim \widetilde\Lambda=p$ and $\widetilde\Lambda$ is composed with a pencil.

    By \cref{propK3Pencil},
    \[
        K_X^3\geq \frac{4p^3}{l(l+p)(3l+4p)}
    \]
    which is a contradiction.

    \medskip

    \textbf{Case 2-2}. $p\geq 2$ and $|\lfloor lf^*K_X \rfloor-S|$ and $f^*\Lambda$ are not composed with the same pencil.

    By taking a higher resolution, we may assume that $\Mov|\lfloor lf^*K_X \rfloor-S|$ is base point free. Let $\widetilde{V}=(\Mov|\lfloor lf^*K_X \rfloor-S|)|_S$. Then $\dim\widetilde{V}\geq 1$. Let $M$ be a general element in $\Mov|\lfloor lf^*K_X \rfloor-S|$. Then $M$ is nef and effective. Let $V_1$ be a 1-dimensional subsystem of $\widetilde{V}$ containing $M|_S$.

    We have
    \[
        |nK_W|\lsgeq |K_W+\lceil (n-l-1)f^*K_X+M\rceil+S|.
    \]

    Since $n\geq l+2$ and $M$ is nef and effective, we have that $(n-l-1)f^*K_X+M$ is nef and big. With Kawamata–Viehweg vanishing theorem,
    \begin{align}\label{eqNKSGeqNM1-DeltaM2L4-2}
        |nK_W||_S\lsgeq |K_S+\lceil (n-l-1)f^*(K_X)|_S \rceil+M|_S|.
    \end{align}

    Since $P_{n-l}(X)>0$, we have $|nK_W|\lsgeq \Mov|\lfloor lf^*K_X \rfloor-S|$. Thus
    \begin{align}\label{eqNKSGeqV-DeltaM2L4-2}
        |nK_W||_S\lsgeq \widetilde{V}\lsgeq V_1.
    \end{align}

    By \cref{propSurfLB}, \cref{eqNKSGeqNM1-DeltaM2L4-2} and \cref{eqNKSGeqV-DeltaM2L4-2}, we have that $|nK_W||_S$ defines a generically finite map.
\end{proof}

\section{Threefolds with large pluricanonical section index}\label{secDelta13}

The following eight propositions are from direct computation of baskets in \cref{tableDelta13} using Reid's Riemann-Roch formula.

\begin{proposition}\label{propDeltaAll}
    Let $X$ be a minimal 3-fold of general type with $\delta(X)\geq 13$. Then $P_{18}(X)\geq 2$, $P_{30}(X)\geq 4$ and $K_X^3\leq \frac{1}{252}$. \qed
\end{proposition}

\begin{proposition}\label{propDelta30leq5}
    Let $X$ be a minimal 3-fold of general type with $\delta(X)\geq 13$. If $P_{30}(X)\leq 5$, then $K_X^3\leq \frac{1}{1170}$. \qed
\end{proposition}

\begin{proposition}\label{propDelta18leq3}
    Let $X$ be a minimal 3-fold of general type with $\delta(X)\geq 13$. If $P_{18}(X)\leq 3$, then $K_X^3\leq \frac{23}{9240}$. \qed
\end{proposition}

\begin{proposition}\label{propDeltaPn15}
    Let $X$ be a minimal 3-fold of general type with $\delta(X)\geq 13$ and $P_{18}(X)\leq 3$. If $K_X^3> \frac{1}{770}$ and $39\leq n\leq 49$, then $P_n(X)\geq 15$. \qed
\end{proposition}

\begin{proposition}\label{propDeltaPn13}
    Let $X$ be a minimal 3-fold of general type with $\delta(X)\geq 13$ and $P_{18}(X)\leq 3$. If $K_X^3> \frac{3}{3080}$ and $39\leq n\leq 49$, then $P_n(X)\geq 13$. \qed
\end{proposition}

\begin{proposition}\label{propDeltaPn11}
    Let $X$ be a minimal 3-fold of general type with $\delta(X)\geq 13$ and $P_{18}(X)\leq 3$. Then $P_n(X)\geq 11$ in each of the following three cases:
    \begin{enumerate}
        \item $K_X^3> \frac{1}{1170}$ and $39\leq n\leq 49$;
        \item $45\leq n\leq 49$;
        \item $n=44$ and $X$ is not of type 7a. \qed
    \end{enumerate}
\end{proposition}

\begin{proposition}\label{propDeltaPn9}
    Let $X$ be a minimal 3-fold of general type with $\delta(X)\geq 13$ and $P_{18}(X)\leq 3$. Then $P_n(X)\geq 9$ in each of the following two cases:
    \begin{enumerate}
        \item $39\leq n\leq 49$ and $X$ is not of type 7a;
        \item $n=42$. \qed
    \end{enumerate}
\end{proposition}

\begin{proposition}\label{propDelta7a}
    Let $X$ be a minimal 3-fold of general type with $\delta(X)\geq 13$. If $X$ is of type 7a, then $K_X^3=\frac{1}{1680}$, $P_6(X)=P_9(X)=1$, $P_{14}(X)=2$, $P_{30}(X)=P_{32}(X)=P_{34}(X)=5$ and $P_{44}(X)=10$. \qed
\end{proposition}

\begin{proposition}\label{propDelta18geq4}
    Let $X$ be a minimal 3-fold of general type with $\delta(X)\geq 13$. If $P_{18}(X)\geq 4$, then $|nK_X|$ defines a generically finite map when $n\geq 38$.
\end{proposition}

\begin{proof}
    Let $f:W\to X$ be a Chen resolution with respect to $\Lambda=|18K_X|$. Let $S$ be a generic irreducible element of $f^*\Lambda$. Let $p=\dim\Lambda$. Then $p\geq 3$. Let $d=\dim\overline{\Phi_{\Lambda}(X)}$.

    \medskip

    \textbf{Case 1}. $d=1$.

    By \cref{thmDelta}, we have $\chi(\OX)\geq 2$ and $q(X)=p_g(X)=0$. Thus $h^2(\OX)>0$. By \cref{propH2OX}, we have that $S$ is not a $(1,0)$-surface.

    Take $m=18$ in \cref{propK3Pencil}, we have
    \begin{align*}
        K_X^3\geq \frac{4p^3}{18(18+p)(3\times 18+4p)}\geq \frac{4\times 3^3}{18(18+3)(3\times 18+4\times 3)} > \frac{1}{252}
    \end{align*}
    which is in contradiction with \cref{propDeltaAll}.

    \medskip

    \textbf{Case 2}. $d\geq 2$.

    By \cref{coroDelta}, $P_{n-18}(X)>0$ since $n\geq 38$.

    Take $m=18$ in \cref{propD2}, we have that $|nK_X|$ defines a generically finite map since $n\geq 38\geq 9+18+2$.
\end{proof}

\begin{proposition}\label{propDelta50}
    Let $X$ be a minimal 3-fold of general type with $\delta(X)\geq 13$. Then $|nK_X|$ defines a generically finite map when $n\geq 50$.
\end{proposition}

\begin{proof}
    By \cref{propDeltaAll}, we have $P_{30}(X)\geq 4$.

    Let $f:W\to X$ be a Chen resolution with respect to $\Lambda=|30K_X|$. Let $S$ be a generic irreducible element of $f^*\Lambda$. Let $p=\dim\Lambda$. Then $p\geq 3$. Let $d=\dim\overline{\Phi_{\Lambda}(X)}$.

    \medskip

    \textbf{Case 1}. $d=1$.

    By \cref{thmDelta}, we have $\chi(\OX)\geq 2$ and $q(X)=p_g(X)=0$. Thus $h^2(\OX)>0$. By \cref{propH2OX}, we have that $S$ is not a $(1,0)$-surface.

    Take $m=30$ in \cref{propK3Pencil}, we have
    \begin{align*}
        K_X^3\geq \frac{4p^3}{30(30+p)(3\times 30+4p)}.
    \end{align*}

    If $P_{30}(X)\geq 6$, then
    \[
        K_X^3\geq \frac{4\times 5^3}{30(30+5)(3\times 30+4\times 5)}> \frac{1}{252}
    \]
    which is in contradiction with \cref{propDeltaAll}.

    If $4\leq P_{30}(X)\leq 5$, then
    \[
        K_X^3\geq \frac{4\times 3^3}{30(30+3)(3\times 30+4\times 3)}> \frac{1}{1170}
    \]
    which is in contradiction with \cref{propDelta30leq5}.

    \medskip

    \textbf{Case 2}. $d\geq 2$.

    By \cref{coroDelta}, $P_{n-30}(X)>0$ since $n\geq 50$.

    Take $m=30$ in \cref{propD2}, we have that $|nK_X|$ defines a generically finite map since $n\geq 50\geq 15+30+2$.
\end{proof}

\begin{proposition}\label{propDelta38}
    Let $X$ be a minimal 3-fold of general type with $\delta(X)\geq 13$. Then $|nK_X|$ defines a generically finite map when $n\geq 38$.
\end{proposition}

\begin{proof}
    \textbf{Case 1}. $n=38$.

    By \cref{propDeltaAll}, $P_{18}(X)\geq 2$. By \cref{coroDelta}, $P_{20}(X)>0$. By \cref{prop2m2}, $|38K_X|$ defines a generically finite map.

    \medskip

    \textbf{Case 2}. $n\geq 50$.

    By \cref{propDelta50}, this case is true.

    \medskip

    \textbf{Case 3}. $P_{18}(X)\geq 4$.

    By \cref{propDelta18geq4}, this case is true.

    \medskip

    \textbf{Case 4}. $P_{18}(X)\leq 3$, $K_X^3> \frac{1}{770}$ and $39\leq n\leq 49$.

    By \cref{propDeltaPn15}, $P_n(X)\geq 15$. Let $p=\lfloor \frac{P_n(X)-1}{2}\rfloor$. Then $p\geq 7$. By \cref{propDelta18leq3},
    \[
        K_X^3\leq \frac{23}{9240}<\frac{4\times 7^3}{49(49+7)(3\times 49+4\times 7)} \leq \frac{4p^3}{n(n+p)(3n+4p)}.
    \]

    By \cref{coroDelta}, $P_{n-18}(X)>0$. Take $m=18$ in \cref{propDeltaM2N3}, we have that $|nK_X|$ defines a generically finite map.

    \medskip

    \textbf{Case 5}. $P_{18}(X)\leq 3$, $\frac{3}{3080}< K_X^3\leq \frac{1}{770}$ and $39\leq n\leq 49$.

    By \cref{propDeltaPn13}, $P_n(X)\geq 13$. Let $p=\lfloor \frac{P_n(X)-1}{2}\rfloor$. Then $p\geq 6$. Thus
    \[
        K_X^3\leq \frac{1}{770}<\frac{4\times 6^3}{49(49+6)(3\times 49+4\times 6)} \leq \frac{4p^3}{n(n+p)(3n+4p)}.
    \]

    By \cref{coroDelta}, $P_{n-18}(X)>0$. Take $m=18$ in \cref{propDeltaM2N3}, we have that $|nK_X|$ defines a generically finite map.

    \medskip

    \textbf{Case 6}. $P_{18}(X)\leq 3$, $\frac{1}{1170}< K_X^3\leq \frac{3}{3080}$ and $39\leq n\leq 49$.

    By \cref{propDeltaPn11}, $P_n(X)\geq 11$. Let $p=\lfloor \frac{P_n(X)-1}{2}\rfloor$. Then $p\geq 5$. Thus
    \[
        K_X^3\leq \frac{3}{3080}<\frac{4\times 5^3}{49(49+5)(3\times 49+4\times 5)} \leq \frac{4p^3}{n(n+p)(3n+4p)}.
    \]

    By \cref{coroDelta}, $P_{n-18}(X)>0$. Take $m=18$ in \cref{propDeltaM2N3}, we have that $|nK_X|$ defines a generically finite map.

    \medskip

    \textbf{Case 7}. $P_{18}(X)\leq 3$, $K_X^3\leq \frac{3}{3080}$ and $45\leq n\leq 49$.

    The proof is the same as that in case 6.

    \medskip

    \textbf{Case 8}. $P_{18}(X)\leq 3$, $K_X^3\leq \frac{3}{3080}$, $n=44$ and $X$ is not of type 7a.

    The proof is the same as that in case 6.

    \medskip

    \textbf{Case 9}. $P_{18}(X)\leq 3$, $K_X^3\leq \frac{1}{1170}$, $39\leq n\leq 43$ and $X$ is not of type 7a.

    By \cref{propDeltaPn9}, $P_n(X)\geq 9$. Let $p=\lfloor \frac{P_n(X)-1}{2}\rfloor$. Then $p\geq 4$. Thus
    \[
        K_X^3\leq \frac{1}{1170}<\frac{4\times 4^3}{43(43+4)(3\times 43+4\times 4)} \leq \frac{4p^3}{n(n+p)(3n+4p)}.
    \]

    By \cref{coroDelta}, $P_{n-18}(X)>0$. Take $m=18$ in \cref{propDeltaM2N3}, we have that $|nK_X|$ defines a generically finite map.

    \medskip

    \textbf{Case 10}. $P_{18}(X)\leq 3$, $K_X^3\leq \frac{1}{1170}$ and $n=42$.

    The proof is the same as that in case 9.

    \medskip

    \textbf{Case 11}. $X$ is of type 7a and $n=44$.

    By \cref{propDelta7a}, $P_n(X)=10$. Let $p=\lfloor \frac{P_n(X)-1}{2}\rfloor$. Then $p= 4$. Thus
    \[
        K_X^3= \frac{1}{1680}<\frac{4\times 4^3}{44(44+4)(3\times 44+4\times 4)} = \frac{4p^3}{n(n+p)(3n+4p)}.
    \]

    By \cref{coroDelta}, $P_{n-14}(X)>0$. Take $m=14$ in \cref{propDeltaM2N3}, we have that $|nK_X|$ defines a generically finite map.

    \medskip

    \textbf{Case 12}. $X$ is of type 7a and $n=43$.

    By \cref{propDelta7a}, $P_9(X)=1$ and $P_{34}(X)=5$. Let $l=34$, $p=P_l(X)-2=3$. Then
    \begin{align*}
        K_X^3= \frac{1}{1680}<\frac{4\times 3^3}{34(34+3)(3\times 34+4\times 3)}
        = \frac{4p^3}{l(l+p)(3l+4p)}.
    \end{align*}

    By \cref{coroDelta}, $P_{n-14}(X)>0$. Take $m=14$ in \cref{propDeltaM2L3}, we have that $|nK_X|$ defines a generically finite map.

    \medskip

    \textbf{Case 13}. $X$ is of type 7a and $n=41$.

    By \cref{propDelta7a}, $P_9(X)=1$ and $P_{32}(X)=5$. Let $l=32$, $p=P_l(X)-2=3$. Then
    \begin{align*}
        K_X^3= \frac{1}{1680}<\frac{4\times 3^3}{32(32+3)(3\times 32+4\times 3)}
        = \frac{4p^3}{l(l+p)(3l+4p)}.
    \end{align*}

    By \cref{coroDelta}, $P_{n-14}(X)>0$. Take $m=14$ in \cref{propDeltaM2L3}, we have that $|nK_X|$ defines a generically finite map.

    \medskip

    \textbf{Case 14}. $X$ is of type 7a and $n=40$.

    By \cref{propDelta7a}, $P_6(X)=1$ and $P_{34}(X)=5$. Let $l=34$, $p=P_l(X)-2=3$. Then
    \begin{align*}
        K_X^3= \frac{1}{1680}<\frac{4\times 3^3}{34(34+3)(3\times 34+4\times 3)}
        = \frac{4p^3}{l(l+p)(3l+4p)}.
    \end{align*}

    By \cref{coroDelta}, $P_{n-14}(X)>0$. Take $m=14$ in \cref{propDeltaM2L3}, we have that $|nK_X|$ defines a generically finite map.

    \medskip

    \textbf{Case 15}. $X$ is of type 7a and $n=39$.

    By \cref{propDelta7a}, $P_9(X)=1$ and $P_{30}(X)=5$. Let $l=30$, $p=P_l(X)-2=3$. Then
    \begin{align*}
        K_X^3= \frac{1}{1680}<\frac{4\times 3^3}{30(30+3)(3\times 30+4\times 3)}
        = \frac{4p^3}{l(l+p)(3l+4p)}.
    \end{align*}

    By \cref{coroDelta}, $P_{n-14}(X)>0$. Take $m=14$ in \cref{propDeltaM2L3}, we have that $|nK_X|$ defines a generically finite map.
\end{proof}

\section{Proof of \texorpdfstring{\cref{thm3m2}}{Theorem 1.2} and \texorpdfstring{\cref{thm38}}{Theorem 1.3}}\label{secThm38}

\begin{theorem}\label{thm2m2Chi1}
    Let $X$ be a minimal 3-fold of general type with $\chi(\OX)=1$ and $P_m(X)\geq 2$. Then $|(2m+2)K_X|$ defines a generically finite map.
\end{theorem}

\begin{proof}
    When $m\geq 2$, by \cref{propPChi1}, $P_{m+2}(X)\geq P_m(X)\geq 2$. When $m=1$, $P_3(X)\geq p_g(X)\geq 2$. By \cref{prop2m2}, $|(2m+2)K_X|$ defines a generically finite map.
\end{proof}

\begin{theorem}\label{thm29Chi1}
    Let $X$ be a minimal 3-fold of general type with $\chi(\OX)=1$. Then $|nK_X|$ defines a generically finite map when $n\geq 29$.
\end{theorem}

\begin{proof}
    By \cref{thmPChi1}, $P_{10}(X)\geq 2$. By \cref{thm2m2Chi1}, $|22K_X|$ defines a generically finite map.

    By \cref{thmPChi1}, $P_n(X)>0$ when $n\geq 7$. Thus $|nK_X|$ defines a generically finite map when $n\geq 22+7$.
\end{proof}

We introduce a modification of formal baskets.

\begin{definition}
    For an integer $N\geq 2$ and a formal basket
    \[
        B=\{n_{b,r}\times (b,r) \mid \gcd(b,r)=1,\ 0<2b\leq r\},
    \]
    we define the level $N$ \textit{truncation}
    \[
        T_N(B)\coloneq \{n_{b,r}\times (b,r) \mid 0<2b\leq r\leq N\}\cup\{n_N\times (1,N)\}
    \]
    where $n_{b,r}=0$ for $\gcd(b,r)>1$, and
    \[
        n_N=\sum_{r\geq N+1,\ 0<2b\leq r}n_{b,r}.
    \]
\end{definition}

\begin{proposition}\label{propBasketK3Chi}
    Let $(B,\chi,\chi_2)$ be a weighted basket. Then
    \begin{align*}
        K^3(T_N(B^{(N-1)}),\chi,\chi_2)\geq K^3(B,\chi,\chi_2)
    \end{align*}
    and
    \begin{align*}
        \chi_m(T_N(B^{(N-1)}),\chi,\chi_2)= \chi_m(B,\chi,\chi_2)
    \end{align*}
    for all $2\leq m\leq N$.
\end{proposition}

\begin{proof}
    By \cref{eqBasketCS}, we can write
    \begin{align*}
        B^{(N-1)}=\{n_{b,r}\times (b,r)\mid 0<2b\leq r\leq N-1\}\cup\{n_{1,r}\times (1,r)\mid r\geq N\}
    \end{align*}
    where $n_{b,r}=0$ for $\gcd(b,r)>1$.

    By Reid's Riemann-Roch formula,
    \begin{align*}
        &\phantom{\;=\;}K^3(T_N(B^{(N-1)}),\chi,\chi_2)-K^3(B^{(N-1)},\chi,\chi_2)\\
        &=\sum_{r\geq N+1}n_{1,r}\left(\frac{1}{N}-\frac{1}{r}\right)\geq 0,
    \end{align*}
    and
    \begin{align*}
        &\phantom{\;=\;}\chi_m(T_N(B^{(N-1)}),\chi,\chi_2)-\chi_m(B^{(N-1)},\chi,\chi_2)\\
        &=\sum_{r\geq N+1}n_{1,r}\left(\frac{m(m-1)(2m-1)}{12}\left(\frac{1}{N}-\frac{1}{r}\right)+ \sum_{j=1}^{m-1}\left(\frac{j(N-j)}{2N}-\frac{j(r-j)}{2r}\right) \right)\\
        &=0
    \end{align*}
    for all $2\leq m\leq N$.

    With \cref{eqBasketK3} and \cref{eqBasketChi}, we are done.
\end{proof}

\begin{proposition}\label{prop101112}
    Let $X$ be a minimal 3-fold of general type with $\chi(\OX)\neq 1$ and $\delta(X)\leq 12$. Then $|nK_X|$ defines a generically finite map when $n\geq 3\delta(X)+2$.
\end{proposition}

\begin{proof}
    Let $\delta=\delta(X)$, $\chi=\chi(\OX)$, $\chi_2=P_2(X)$.

    For all positive integers $a,b$, let
    \begin{align*}
        C(a,b)= \frac{4b^3}{a(a+b)(3a+4b)}.
    \end{align*}

    By \cref{thmIrr5}, we may assume $q(X)=0$.

    If $p_g(X)>0$, then $P_m(X)>0$ for all positive $m$. Then $|(2\delta+2)K_X|$ defines a generically finite map by \cref{prop2m2}. Then $|nK_X|$ defines a generically finite map since $n\geq 2\delta+2$.
    
    Thus we may assume $p_g(X)=0$. Then $\chi=1+h^2(\OX)\geq 1$. Since $\chi\neq 1$, we have $\chi\geq 2$.

    By \cref{propMasekChi3}, we may assume $\chi\leq 2$. Thus
    \begin{align}\label{eq101112-chiE2}
        \chi= 2.
    \end{align}

    Thus $h^2(\OX)=1$. Let $\Lambda$ be a 1-dimensional subsystem of $|\delta K_X|$. By \cref{propH2OX}, we have that a general fiber of the induced fibration from $\Phi_{\Lambda}$ is not a $(1,0)$-surface. By \cref{propK3Pencil},
    \begin{align}\label{eq101112-KX3Geq}
        K_X^3\geq C(\delta, 1).
    \end{align}

    For all $1\leq m\leq \delta-1$, we have
    \begin{align}\label{eq101112-Pm1}
        P_m(X)\leq 1.
    \end{align}

    We always have that if $P_a(X)>0$ and $P_b(X)>0$, then
    \begin{align}\label{eq101112-semigroup}
        P_{a+b}(X)\geq P_a(X)+P_b(X)-1.
    \end{align}

    Let $N_\delta=\lfloor (2\sqrt{2}+1)(\delta+1)\rfloor$.

    \medskip

    \textbf{Case 1}. $n\geq N_\delta+1$.

    Take $m=\delta$ in \cref{prop2sqrt2}.

    \medskip

    \textbf{Case 2}. $n=3\delta+2$.

    Take $l=1, m=\delta$ in \cref{prop3m2}.

    \medskip

    \textbf{Case 3}. $n=3\delta+3$.

    Take $l=3, m=\delta$ in \cref{prop3m3}.

    \medskip

    \textbf{Case 4}. $3\delta+4\leq n\leq N_\delta$ and $\delta\leq 2$.

    Since $p_g(X)=0$, we have $\delta\geq 2$. Thus $\delta=2$. Thus $n=10$ or $11$.

    By \cref{thm6K}, $|6K_X|$ defines a generically finite map. By \cref{thmNV}, since $n-6\geq 2\delta$, we have $P_{n-6}(X)>0$. Thus $|nK_X|$ defines a generically finite map.

    \medskip

    \textbf{Case 5}. $3\delta+4\leq n\leq N_\delta$ and $\delta\geq 3$.

    By \cref{thmNV} and \cref{propMasek3}, we may assume that if $P_m(X)\geq 3$ and $3\delta+4-m\geq 2\delta$, then
    \begin{align}\label{eq101112-masekAll}
        3\delta+4\leq \frac{P_m(X)+2}{P_m(X)-1}m+4.
    \end{align}

    Take $m=\delta$ in \cref{eq101112-masekAll}, we have $P_\delta(X)\leq 2$. Thus
    \begin{align}\label{eq101112-Pdelta2}
        P_\delta(X)= 2.
    \end{align}

    By \cref{eq101112-masekAll} and $\delta\geq 3$, we have that if $m\leq 7$, then
    \begin{align}\label{eq101112-Pm7}
        P_m(X)\leq\max\left\{2, \frac{3\delta+2m}{3\delta-m}\right\}.
    \end{align}

    Let $B$ be the Reid basket of $X$. Then
    \begin{align}\label{eq101112-k3B}
        K^3(B,\chi,\chi_2)=K_X^3
    \end{align}
    and
    \begin{align}\label{eq101112-chimB}
        \chi_m(B,\chi,\chi_2)=P_m(X)
    \end{align}
    for all $m\geq 2$.

    Let $\widetilde{B}=T_{N_\delta}(B^{(N_\delta-1)})$. Let $K^3=K^3(\widetilde{B},\chi,\chi_2)$. Let $\chi_m=\chi_m(\widetilde{B},\chi,\chi_2)$ for all $m\geq 3$.

    By \cref{propBasketK3Chi}, \cref{eq101112-k3B} and \cref{eq101112-chimB},
    \begin{align}\label{eq101112-k3}
        K^3 \geq K_X^3
    \end{align}
    and
    \begin{align}\label{eq101112-chim}
        \chi_m = P_m(X)
    \end{align}
    for all $2\leq m\leq N_\delta$.

    Write
    \begin{align*}
        \widetilde{B}^{(0)}=\{n^0_{1,r}\times (1,r)\mid 2\leq r\leq N_\delta \}.
    \end{align*}

    By \cite[(3.6), (3.9), (3.10)]{EXP1},
    \begin{gather*}
        n^0_{1,2}=5\chi+6\chi_2-4\chi_3+\chi_4,\\
        n^0_{1,3}=4\chi+2\chi_2+2\chi_3-3\chi_4+\chi_5,\\
        \sum_{r\geq4}n^0_{1,r}=\chi-3\chi_2+\chi_3+2\chi_4-\chi_5,\\
        n^0_{1,5}+2\sum_{r\geq6}n^0_{1,r}=\chi_4+\chi_5+\chi_6-3\chi_2-\chi_3-\chi_7.
    \end{gather*}

    With \cref{eq101112-chim}, \cref{eq101112-Pm7} and \cref{eq101112-chiE2}, there are finite possibilities of $\widetilde{B}^{(0)}$. Thus there are finite possibilities of $\widetilde{B}$.

    After excluding some non-geometric possibilities using \cref{eq101112-chim}, \cref{eq101112-k3}, \cref{eq101112-Pm7}, \cref{eq101112-Pdelta2}, \cref{eq101112-masekAll}, \cref{eq101112-semigroup}, \cref{eq101112-Pm1} and \cref{eq101112-KX3Geq}, there are 87652, 34710, 5074, 1843, 566, 513, 79, 111, 6 or 35 surviving possibilities of $(\widetilde{B},\chi,\chi_2)$ when $\delta=3,4,5,6,7,8,9,10,11$ or $12$ respectively.

    By direct computation of these finite possibilities using Reid's Riemann-Roch formula, we have that either
    \begin{align}\label{eq101112-2m2}
        \chi_{\delta+2}>0,\ \chi_{n-2\delta-2}>0.
    \end{align}
    or
    \begin{align}\label{eq101112-m2n3}
        \chi_n\geq 3,\ \chi_{n-\delta}>0,\ K^3<C\left(n,\left\lfloor\frac{\chi_n-1}{2} \right\rfloor\right).
    \end{align}

    The complete computation is shown in \cref{lst101112}.

    When \cref{eq101112-2m2} is satisfied, by \cref{prop2m2} and \cref{eq101112-chim}, we have that $|nK_X|\lsgeq |(2\delta+2)K_X|$ defines a generically finite map.

    When \cref{eq101112-m2n3} is satisfied, by \cref{propDeltaM2N3}, \cref{eq101112-chim}, \cref{eq101112-k3}, \cref{eq101112-chiE2} and $n\geq 2\delta+3$, we have that $|nK_X|$ defines a generically finite map.
\end{proof}

\begin{proof}[Proof of \cref{thm3m2}]
    By \cref{propDelta38} and \cref{prop101112}, we only need to consider the case that $\chi(\OX)=1$. Let $\delta=\delta(X)$.

    \medskip

    \textbf{Case 1}. $n=3\delta+2$.

    Take $l=1, m=\delta$ in \cref{prop3m2}.

    \medskip

    \textbf{Case 2}. $n=3\delta+3$.

    Take $l=3, m=\delta$ in \cref{prop3m3}.

    \medskip

    \textbf{Case 3}. $n\geq 3\delta+4$.

    By \cref{thm2m2Chi1}, $|(2\delta+2)K_X|$ defines a generically finite map. Let $k=n-2\delta-2$. Then we only need to prove that $P_k(X)>0$.

    \medskip

    \textbf{Case 3-1}. $k\geq 7$.

    By \cref{thmPChi1}, $P_k(X)>0$.

    \medskip

    \textbf{Case 3-2}. $\delta\leq 2$.

    By \cref{thmNV}, since $k\geq \delta+2\geq 2\delta$, we have $P_k(X)>0$.

    \medskip

    \textbf{Case 3-3}. $\delta\geq 3$ and $k=\delta+2$.

    By \cref{propPChi1}, $P_k(X)\geq P_\delta(X)\geq 2$.

    \medskip

    \textbf{Case 3-4}. $\delta\geq 3$ and $\delta+3\leq k\leq 6$.

    In this case, we have $\delta=3$ and $k=6$. Thus $P_k(X)\geq 2P_\delta(X)-1\geq 3$.
\end{proof}

\begin{proof}[Proof of \cref{thm38}]
    By \cref{propDelta38}, \cref{thm29Chi1} and \cref{prop101112}, we are done.
\end{proof}

The bound 38 is substantially below the known birationality bound 57 (\cite[Theorem 1.1]{Chen2018}, \cite[Theorem 6.2]{EXP3}). It is natural to ask the following question.

\begin{question}
    Is 38 the smallest integer $m$ such that $|nK_X|$ defines a generically finite map for all $n\geq m$ and all minimal 3-folds $X$ of general type?
\end{question}

The following example is the ``worst'' known example.

\begin{example}[{\cite[II.5.1]{IF2000}}]
    $X_{46}$ in $\bP(4,5,6,7,23)$ is a minimal 3-fold of general type such that $|15K|$ does not define a generically finite map and $|nK|$ defines a generically finite map for all $n\geq 16$.
\end{example}

\section{Appendix}

\begin{singlespace}
\renewcommand{\arraystretch}{0.7}
\begin{xltabular}{\textwidth}{R B C C C C C}
\caption{Possible baskets with $\delta\geq 13$}\label{tableDelta13}\\

\toprule
Type & Basket & $K^3$ & $\chi$ & $P_{18}$ & $P_{30}$ & $\delta$ \\
\midrule
\endfirsthead

\toprule
Type & Basket & $K^3$ & $\chi$ & $P_{18}$ & $P_{30}$ & $\delta$ \\
\midrule
\endhead

\midrule
\multicolumn{7}{r}{Continued on next page}\\
\endfoot

\bottomrule
\endlastfoot

2a & $\{4\times(1,2),(4,9),(2,5),(5,13),3\times(1,3),2\times(1,4)\}$ & $1/1170$ & 2 & 2 & 4 & 18\\
\midrule

3 & $\{6\times(1,2),(5,11),4\times(2,5),(3,8),4\times(1,3),(2,7),2\times(1,4)\}$ & $23/9240$ & 3 & 3 & 12 & 15\\
5.1 & $\{7\times(1,2),(4,9),3\times(2,5),(5,13),4\times(1,3),(3,11),(1,4)\}$ & $61/25740$ & 3 & 4 & 11 & 15\\
5.2 & $\{7\times(1,2),(4,9),2\times(2,5),(7,18),4\times(1,3),(3,11),(1,4)\}$ & $1/660$ & 3 & 4 & 8 & 15\\
5.3 & $\{7\times(1,2),(4,9),(2,5),(9,23),4\times(1,3),(3,11),(1,4)\}$ & $47/45540$ & 3 & 4 & 6 & 15\\
5a & $\{7\times(1,2),(4,9),(11,28),4\times(1,3),(3,11),(1,4)\}$ & $1/1386$ & 3 & 4 & 5 & 15\\
53a & $\{3\times(1,2),(4,9),2\times(2,5),(5,13),3\times(1,3),(1,5)\}$ & $1/1170$ & 2 & 3 & 6 & 15\\
\midrule

4.5 & $\{7\times(1,2),(4,9),4\times(2,5),(5,14),2\times(1,3),(2,7),2\times(1,4)\}$ & $1/630$ & 3 & 2 & 8 & 14\\
5.4 & $\{7\times(1,2),(4,9),4\times(2,5),(3,8),4\times(1,3),(4,15)\}$ & $1/360$ & 3 & 4 & 14 & 14\\
6 & $\{9\times(1,2),2\times(3,7),(2,5),(4,11),4\times(1,3),2\times(2,7),(1,5)\}$ & $1/462$ & 3 & 3 & 11 & 14\\
7 & $\{5\times(1,2),(4,9),(3,7),5\times(1,3),(2,7),(1,5)\}$ & $1/630$ & 2 & 3 & 8 & 14\\
7a & $\{5\times(1,2),(7,16),5\times(1,3),(2,7),(1,5)\}$ & $1/1680$ & 2 & 2 & 5 & 14\\
10 & $\{8\times(1,2),(4,9),(3,7),2\times(3,8),5\times(1,3),(2,7),(1,4),(1,5)\}$ & $1/630$ & 3 & 3 & 8 & 14\\
11 & $\{9\times(1,2),2\times(3,7),(3,8),(4,11),3\times(1,3),(3,10),(1,4),(1,5)\}$ & $3/3080$ & 3 & 2 & 7 & 14\\
12 & $\{9\times(1,2),(4,9),(2,5),2\times(3,8),4\times(1,3),2\times(2,7),(1,5)\}$ & $1/252$ & 3 & 5 & 18 & 14\\
12.1 & $\{9\times(1,2),(4,9),(5,13),(3,8),4\times(1,3),2\times(2,7),(1,5)\}$ & $67/32760$ & 3 & 4 & 10 & 14\\
12a & $\{9\times(1,2),(4,9),(8,21),4\times(1,3),2\times(2,7),(1,5)\}$ & $1/630$ & 3 & 4 & 8 & 14\\
17 & $\{9\times(1,2),2\times(3,7),2\times(4,11),3\times(1,3),(2,7),(1,4),(1,5)\}$ & $3/1540$ & 3 & 3 & 11 & 14\\
18 & $\{9\times(1,2),2\times(3,7),(3,8),(4,11),4\times(1,3),(3,11),(1,5)\}$ & $23/9240$ & 3 & 4 & 13 & 14\\
18b & $\{9\times(1,2),2\times(3,7),(7,19),4\times(1,3),(3,11),(1,5)\}$ & $83/43890$ & 3 & 4 & 11 & 14\\
20 & $\{7\times(1,2),2\times(4,9),(2,5),(3,8),6\times(1,3),(2,7),(1,4),(1,5)\}$ & $1/504$ & 3 & 4 & 10 & 14\\
21 & $\{6\times(1,2),(4,9),(3,8),3\times(1,3),(3,10),(1,5)\}$ & $1/360$ & 2 & 4 & 14 & 14\\
23 & $\{8\times(1,2),(4,9),(3,7),(2,5),(4,11),4\times(1,3),(3,10),(1,4),(1,5)\}$ & $19/13860$ & 3 & 3 & 9 & 14\\
28 & $\{5\times(1,2),(5,11),(3,8),4\times(1,3),(2,7),(1,5)\}$ & $23/9240$ & 2 & 4 & 13 & 14\\
29 & $\{6\times(1,2),(4,9),(4,11),3\times(1,3),(2,7),(1,5)\}$ & $13/3465$ & 2 & 5 & 18 & 14\\
29.1 & $\{6\times(1,2),(4,9),(5,14),2\times(1,3),(2,7),(1,5)\}$ & $1/630$ & 2 & 3 & 9 & 14\\
30 & $\{7\times(1,2),(5,11),(3,7),(2,5),(4,11),5\times(1,3),(2,7),(1,4),(1,5)\}$ & $1/924$ & 3 & 3 & 8 & 14\\
31 & $\{7\times(1,2),(5,11),(3,7),(2,5),(3,8),6\times(1,3),(3,11),(1,5)\}$ & $1/616$ & 3 & 4 & 10 & 14\\
32 & $\{8\times(1,2),(4,9),(3,7),(2,5),(4,11),5\times(1,3),(3,11),(1,5)\}$ & $2/693$ & 3 & 5 & 15 & 14\\
32a & $\{8\times(1,2),(7,16),(2,5),(4,11),5\times(1,3),(3,11),(1,5)\}$ & $1/528$ & 3 & 4 & 12 & 14\\
33 & $\{5\times(1,2),2\times(3,7),(3,8),(1,3),(3,10),(2,7)\}$ & $1/840$ & 2 & 2 & 7 & 14\\
34 & $\{7\times(1,2),(4,9),(3,7),2\times(2,5),(3,8),3\times(1,3),3\times(2,7)\}$ & $1/360$ & 3 & 4 & 13 & 14\\
34a & $\{7\times(1,2),(7,16),2\times(2,5),(3,8),3\times(1,3),3\times(2,7)\}$ & $1/560$ & 3 & 3 & 10 & 14\\
35 & $\{5\times(1,2),2\times(3,7),(4,11),(1,3),2\times(2,7)\}$ & $1/462$ & 2 & 3 & 11 & 14\\
36 & $\{4\times(1,2),(4,9),(3,7),(2,5),2\times(1,3),(3,10),(2,7)\}$ & $1/630$ & 2 & 3 & 9 & 14\\
36a & $\{4\times(1,2),(7,16),(2,5),2\times(1,3),(3,10),(2,7)\}$ & $1/1680$ & 2 & 2 & 6 & 14\\
36b & $\{4\times(1,2),(4,9),(3,7),(2,5),2\times(1,3),(5,17)\}$ & $4/5355$ & 2 & 2 & 6 & 14\\
37 & $\{6\times(1,2),2\times(4,9),3\times(2,5),4\times(1,3),3\times(2,7)\}$ & $1/315$ & 3 & 5 & 15 & 14\\
38 & $\{3\times(1,2),(5,11),(3,7),(2,5),3\times(1,3),2\times(2,7)\}$ & $1/770$ & 2 & 3 & 8 & 14\\
39 & $\{7\times(1,2),(4,9),(3,7),(2,5),2\times(3,8),2\times(1,3),(3,10),(2,7),(1,4)\}$ & $1/630$ & 3 & 3 & 9 & 14\\
42 & $\{6\times(1,2),(5,11),(3,7),(2,5),2\times(3,8),3\times(1,3),2\times(2,7),(1,4)\}$ & $1/770$ & 3 & 3 & 8 & 14\\
43 & $\{7\times(1,2),(4,9),(3,7),(2,5),(3,8),(4,11),2\times(1,3),2\times(2,7),(1,4)\}$ & $71/27720$ & 3 & 4 & 13 & 14\\
43.1 & $\{7\times(1,2),(7,16),(2,5),(3,8),(4,11),2\times(1,3),2\times(2,7),(1,4)\}$ & $29/18480$ & 3 & 3 & 10 & 14\\
43c & $\{7\times(1,2),(7,16),(2,5),(7,19),2\times(1,3),2\times(2,7),(1,4)\}$ & $31/31920$ & 3 & 3 & 8 & 14\\
43.2 & $\{7\times(1,2),(4,9),(3,7),(2,5),(7,19),2\times(1,3),2\times(2,7),(1,4)\}$ & $47/23940$ & 3 & 4 & 11 & 14\\
44 & $\{7\times(1,2),(4,9),(3,7),(2,5),2\times(3,8),3\times(1,3),(2,7),(3,11)\}$ & $43/13860$ & 3 & 5 & 15 & 14\\
44.1 & $\{7\times(1,2),(4,9),(3,7),(5,13),(3,8),3\times(1,3),(2,7),(3,11)\}$ & $85/72072$ & 3 & 4 & 7 & 14\\
44.2 & $\{7\times(1,2),(4,9),(3,7),(2,5),2\times(3,8),3\times(1,3),(5,18)\}$ & $1/420$ & 3 & 5 & 12 & 14\\
44.3 & $\{7\times(1,2),(7,16),(2,5),2\times(3,8),3\times(1,3),(2,7),(3,11)\}$ & $13/6160$ & 3 & 4 & 12 & 14\\
44c & $\{7\times(1,2),(7,16),(2,5),2\times(3,8),3\times(1,3),(5,18)\}$ & $1/720$ & 3 & 4 & 9 & 14\\
45 & $\{3\times(1,2),2\times(4,9),(3,8),3\times(1,3),(2,7),(1,4)\}$ & $1/504$ & 2 & 4 & 10 & 14\\
46 & $\{6\times(1,2),2\times(4,9),2\times(2,5),(3,8),3\times(1,3),(3,10),(2,7),(1,4)\}$ & $1/504$ & 3 & 4 & 11 & 14\\
46b & $\{6\times(1,2),2\times(4,9),2\times(2,5),(3,8),3\times(1,3),(5,17),(1,4)\}$ & $7/6120$ & 3 & 3 & 8 & 14\\
48 & $\{4\times(1,2),(4,9),(3,7),(4,11),(1,3),(3,10),(1,4)\}$ & $19/13860$ & 2 & 3 & 9 & 14\\
49 & $\{5\times(1,2),(5,11),(4,9),2\times(2,5),(3,8),4\times(1,3),2\times(2,7),(1,4)\}$ & $47/27720$ & 3 & 4 & 10 & 14\\
49a & $\{5\times(1,2),(9,20),2\times(2,5),(3,8),4\times(1,3),2\times(2,7),(1,4)\}$ & $1/840$ & 3 & 4 & 8 & 14\\
50 & $\{6\times(1,2),2\times(4,9),2\times(2,5),(4,11),3\times(1,3),2\times(2,7),(1,4)\}$ & $41/13860$ & 3 & 5 & 15 & 14\\
51 & $\{6\times(1,2),2\times(4,9),2\times(2,5),(3,8),4\times(1,3),(2,7),(3,11)\}$ & $97/27720$ & 3 & 6 & 17 & 14\\
51.1 & $\{6\times(1,2),2\times(4,9),(2,5),(5,13),4\times(1,3),(2,7),(3,11)\}$ & $71/45045$ & 3 & 5 & 9 & 14\\
52 & $\{4\times(1,2),(3,7),2\times(2,5),2\times(3,8),2\times(1,3),(1,5)\}$ & $1/420$ & 2 & 3 & 12 & 14\\
53 & $\{3\times(1,2),(4,9),3\times(2,5),(3,8),3\times(1,3),(1,5)\}$ & $1/360$ & 2 & 4 & 14 & 14\\
54 & $\{2\times(1,2),2\times(3,7),3\times(2,5),(3,8),(1,3),(2,7)\}$ & $1/840$ & 2 & 2 & 7 & 14\\
56 & $\{(1,2),(4,9),(3,7),4\times(2,5),2\times(1,3),(2,7)\}$ & $1/630$ & 2 & 3 & 9 & 14\\
58 & $\{4\times(1,2),(4,9),(3,7),4\times(2,5),2\times(3,8),2\times(1,3),(2,7),(1,4)\}$ & $1/630$ & 3 & 3 & 9 & 14\\
59 & $\{2\times(1,2),2\times(3,7),2\times(2,5),(3,8),(4,11),(1,4)\}$ & $3/3080$ & 2 & 2 & 7 & 14\\
60 & $\{3\times(1,2),2\times(4,9),5\times(2,5),(3,8),3\times(1,3),(2,7),(1,4)\}$ & $1/504$ & 3 & 4 & 11 & 14\\
62 & $\{(1,2),(4,9),(3,7),3\times(2,5),(4,11),(1,3),(1,4)\}$ & $19/13860$ & 2 & 3 & 9 & 14\\
\midrule

41 & $\{5\times(1,2),(4,9),2\times(3,8),(1,3),2\times(2,7)\}$ & $1/252$ & 2 & 5 & 18 & 13\\
\end{xltabular}
\renewcommand{\arraystretch}{1}
\end{singlespace}

\begin{singlespace}
\begin{lstlisting}[
  language=Python,
  caption={Computation for \cref{prop101112}},
  label={lst101112}
]
# Require Python 3.8 or later; no third-party packages are needed.
from fractions import Fraction
from functools import lru_cache
from itertools import product
from math import gcd, isqrt

DELTAS = tuple(range(3, 13))
CHI = 2

@lru_cache(maxsize=None)
def C(a, b):
    return Fraction(4 * b**3, a * (a + b) * (3 * a + 4 * b))

@lru_cache(maxsize=None)
def correction(b, r, m):  # l_(b,r)(m) in Reid's Riemann-Roch formula
    total = Fraction(0)
    for j in range(1, m):
        residue = (j * b) % r
        total += Fraction(residue * (r - residue), 2 * r)
    return total

@lru_cache(maxsize=None)
def volume(basket, chi, chi2):
    return (
        6 * chi
        + 2 * chi2
        - sum(number * Fraction(b * (r - b), r) for b, r, number in basket)
    )

@lru_cache(maxsize=None)
def Pm(basket, chi, chi2, m):
    if m == 0:
        return 1
    if m == 1:
        return 0
    return (
        Fraction(m * (m - 1) * (2 * m - 1), 12) * volume(basket, chi, chi2)
        - (2 * m - 1) * chi
        + sum(number * correction(b, r, m) for b, r, number in basket)
    ).numerator

def dict_to_set(basket):
    return frozenset((b, r, number) for (b, r), number in basket.items() if number > 0)

def masek_max_Pm(delta, m):
    return None if m > delta + 4 else max(2, (3 * delta + 2 * m) // (3 * delta - m))

def Pm_range(delta, m):
    if m < delta:
        return range(0, 2)
    if m == delta:
        return (2,)
    maximum = masek_max_Pm(delta, m)
    if maximum is None:
        raise RuntimeError()
    return range(0, maximum + 1)

def check_volume(delta, basket, chi2):
    return volume(basket, CHI, chi2) >= C(delta, 1)

def check_Pm(delta, m, value):
    if value < 0:
        return False
    if m < delta:
        return value <= 1
    if m == delta:
        return value == 2
    maximum = masek_max_Pm(delta, m)
    return maximum is None or value <= maximum

def check_semigroup(P, m):
    return all(
        P[m] >= P[a] + P[m - a] - 1
        for a in range(1, m // 2 + 1)
        if P[a] > 0 and P[m - a] > 0
    )

def initial_baskets(N, delta):
    result = set()
    ranges = [Pm_range(delta, m) for m in range(2, 8)]
    for p2, p3, p4, p5, p6, p7 in product(*ranges):
        n12 = 5 * CHI + 6 * p2 - 4 * p3 + p4
        n13 = 4 * CHI + 2 * p2 + 2 * p3 - 3 * p4 + p5
        tail_size = CHI - 3 * p2 + p3 + 2 * p4 - p5
        tail_weight = p4 + p5 + p6 - 3 * p2 - p3 - p7
        if min(n12, n13, tail_size, tail_weight) < 0:
            continue
        for s in range(tail_weight // 2 + 1):
            n15 = tail_weight - 2 * s
            n14 = tail_size - n15 - s
            if min(n14, n15) < 0:
                continue
            dictionary = {(1, 2): n12, (1, 3): n13, (1, 4): n14, (1, 5): n15}
            if s == 0:
                result.add((dict_to_set(dictionary), p2))
            else:

                def distribute_tail(start, remaining, current):
                    if remaining == 0:
                        dict_with_tail = dictionary.copy()
                        dict_with_tail.update(current)
                        result.add((dict_to_set(dict_with_tail), p2))
                        return
                    for r in range(start, N + 1):
                        current[(1, r)] = current.get((1, r), 0) + 1
                        distribute_tail(r, remaining - 1, current)
                        current[(1, r)] -= 1

                distribute_tail(6, s, {})
    return result

@lru_cache(maxsize=None)
def prime_packings(r):  # all prime packings of level r
    result = []
    for b in range(2, (r - 1) // 2 + 1):
        if gcd(b, r) != 1:
            continue
        r1 = pow(b, -1, r)
        b1 = (b * r1 - 1) // r
        result.append(((b1, r1), (b - b1, r - r1), (b, r)))
    return tuple(result)

def pack_level(basket, level):  # apply prime packings of the level
    family = {basket}
    for parent1, parent2, child in prime_packings(level):
        next_family = set()
        for current in family:
            dictionary = {(b, r): number for b, r, number in current}
            maximum = min(dictionary.get(parent1, 0), dictionary.get(parent2, 0))
            for amount in range(maximum + 1):
                packed = dictionary.copy()
                packed[parent1] = packed.get(parent1, 0) - amount
                packed[parent2] = packed.get(parent2, 0) - amount
                packed[child] = packed.get(child, 0) + amount
                next_family.add(dict_to_set(packed))
        family = next_family
    return family

def enumerate_baskets(delta):
    d1 = delta + 1
    N = d1 + isqrt(8 * d1 * d1)
    P_dict = {}
    for basket, chi2 in initial_baskets(N, delta):
        if not check_volume(delta, basket, chi2):
            continue
        P = tuple(Pm(basket, CHI, chi2, m) for m in range(6))
        if all(check_Pm(delta, m, P[m]) and check_semigroup(P, m) for m in range(2, 6)):
            P_dict[(basket, chi2)] = P
    for r in range(5, N):  # packings after level r do not change P_{r+1}
        candidates = {}
        for (basket, chi2), P in P_dict.items():
            for packed in pack_level(basket, r):
                candidates[(packed, chi2)] = P
        new_P_dict = {}
        for (basket, chi2), P in candidates.items():
            if not check_volume(delta, basket, chi2):
                continue
            new_P = P + (Pm(basket, CHI, chi2, r + 1),)
            if check_Pm(delta, r + 1, new_P[r + 1]) and check_semigroup(new_P, r + 1):
                new_P_dict[(basket, chi2)] = new_P
        P_dict = new_P_dict
    return N, P_dict

def solved(delta, V, P, n):
    return (P[delta + 2] > 0 and P[n - 2 * delta - 2] > 0) or (
        P[n] >= 3 and P[n - delta] > 0 and V < C(n, (P[n] - 1) // 2)
    )

def main():
    all_data = []
    print("delta  N_delta  baskets")
    for delta in DELTAS:
        N, baskets = enumerate_baskets(delta)
        print(delta, N, len(baskets))
        data = []
        for (basket, chi2), P in baskets.items():
            data.append((basket, volume(basket, CHI, chi2), P))
        all_data.append((delta, N, data))
    print("\ndelta  n  exceptions")
    for delta, N, data in all_data:
        for n in range(3 * delta + 4, N + 1):
            exceptions = [
                item for item in data if not solved(delta, item[1], item[2], n)
            ]
            print(delta, n, len(exceptions))

if __name__ == "__main__":
    main()
\end{lstlisting}
\end{singlespace}

\section*{Acknowledgements}

The author would like to express his sincere gratitude to his advisor, Professor Meng Chen, for his support, encouragement, and valuable guidance.



\end{document}